\documentclass[11pt,reqno]{amsart}
\usepackage[letterpaper,margin=1in]{geometry}                
\usepackage{amssymb,amsmath,amsfonts,stmaryrd}
\usepackage{xcolor}
\usepackage[colorlinks,
             linkcolor=blue,
             citecolor=green!70!black,
             pdfproducer={pdfLaTeX},
             pdfpagemode=None,
             bookmarksopen=true
             bookmarksnumbered=true]{hyperref}
\usepackage{graphicx}
\usepackage{multicol}

\usepackage{tikz}
\usetikzlibrary{decorations.markings,arrows,decorations.pathreplacing}
\usepackage{arydshln}

\newcommand\arXiv[1]{\href{http://arxiv.org/abs/#1}{\nolinkurl{arXiv:#1}}}
\newcommand\MRnumber[1]{\href{http://www.ams.org/mathscinet-getitem?mr=#1}{\nolinkurl{MR#1}}}
\newcommand\DOI[1]{\href{http://dx.doi.org/#1}{\nolinkurl{DOI:#1}}}
\newcommand\MAILTO[1]{\href{mailto:#1}{\nolinkurl{#1}}}

\newtheorem*{mainthm}{Main Theorem}

\newtheorem{dummy}{Dummy}[subsection]

\newtheorem{lemma}[dummy]{Lemma}
\newtheorem{proposition}[dummy]{Proposition}
\newtheorem{corollary}[dummy]{Corollary}
\newtheorem{theorem}[dummy]{Theorem}
\newtheorem{definition}[dummy]{Definition}

\theoremstyle{definition}

\newtheorem{remark}[dummy]{Remark}

\usepackage{dsfont}
\renewcommand\mathbb\mathds

\newcommand\bC{\mathbb C}

\newcommand\bK{\mathbb K}
\newcommand\bL{\mathbb L}

\newcommand\bP{\mathbb P}

\newcommand\bR{\mathbb R}

\newcommand\bZ{\mathbb Z}

\newcommand\cA{\mathcal A}
\newcommand\cB{\mathcal B}
\newcommand\cC{\mathcal C}

\newcommand\cE{\mathcal E}
\newcommand\cF{\mathcal F}
\newcommand\cG{\mathcal G}

\newcommand\cM{\mathcal M}
\newcommand\cN{\mathcal N}
\newcommand\cO{\mathcal O}

\newcommand\cR{\mathcal R}
\newcommand\cS{\mathcal S}
\newcommand\cT{\mathcal T}

\newcommand\cW{\mathcal W}
\newcommand\cX{\mathcal X}

\newcommand\cZ{\mathcal Z}

\newcommand\rB{\mathrm B}

\newcommand\rS{\mathrm S}

\newcommand\rU{\mathrm U}

\usepackage[mathscr]{euscript}

\DeclareMathOperator\homology{H}
\renewcommand\H{\homology}

\renewcommand\d{\mathrm d}

\newcommand\s{\mathbf{S}}

\newcommand\W{\underline{\cW}}
\newcommand\SW{\underline{\rS\cW}}

\newcommand\longto\longrightarrow
\newcommand\mono\hookrightarrow
\newcommand\epi\twoheadrightarrow
\newcommand\isom{\overset\sim\to}
\newcommand\<\langle
\renewcommand\>\rangle
\newcommand\sminus\smallsetminus
\newcommand\condense{\mathrel{\,\hspace{.75ex}\joinrel\rhook\joinrel\hspace{-.75ex}\joinrel\rightarrow}}

\newcommand\one{\mathbb{1}}

\newcommand\id{\mathrm{id}}

\newcommand\Spin{\mathrm{Spin}}

\DeclareMathOperator\Aut{Aut}

\DeclareMathOperator\Cliff{Cliff}

\newcommand\Sq{\mathrm{Sq}}
\DeclareMathOperator\SH{SH}

\DeclareMathOperator\End{End}

\newcommand\pt{\mathrm{pt}}

\renewcommand\Vec{\cat{Vec}}

\newcommand\SVec{\cat{SVec}}
\newcommand\Rep{\cat{Rep}}

\newcommand\define[1]{\emph{#1}}
\newcommand\cat[1]{\mathbf{#1}}

\title{(3+1)D topological orders with only a $\bZ_2$-charged particle}
\author{Theo Johnson-Freyd}
\thanks{I thank Davide Gaiotto, Tian Lan, Dmitri Nykshych, David Reutter, Chong Wang,  Matthew Yu, and Hao Zheng for conversations related to this project.
 Research at the Perimeter Institute is supported by the Government of Canada through Industry Canada and by the Province of Ontario through the Ministry of Economic Development and Innovation. 
 The Perimeter Institute is in the Haldimand Tract, land promised to the Six Nations. Dalhousie University is in Mi`kma`ki, the ancestral and unceded territory of the Mi`kmaq. 
 \\[6pt]
\textsc{Perimeter Institute for Theoretical Physics, Waterloo, Ontario}, \MAILTO{theojf@pitp.ca}.\\ 
\textsc{Department of Mathematics, Dalhousie University, Halifax, Nova Scotia}, \MAILTO{theojf@dal.ca}.
}

\begin{document}
\begin{abstract}
There is exactly one bosonic (3+1)-dimensional topological order whose only nontrivial particle is an emergent boson: pure $\bZ_2$ gauge theory. There are exactly two (3+1)-dimensional topological orders whose only nontrivial particle is an emergent fermion: pure ``spin-$\bZ_2$'' gauge theory, in which the dynamical field is a spin structure; and an anomalous version thereof. I give three proofs of this classification, varying from hands-on to abstract.  Along the way, I provide a detailed study of the braided fusion $2$-category $\cZ_{(1)}(\Sigma \SVec)$ of string and particle operators in pure spin-$\bZ_2$ gauge theory.
\end{abstract}
\maketitle


\section{Introduction}

\subsection{Statement of the main result}

A \define{(3+1)D topological order}, if it has a single local vacuum, is described algebraically by a nondegenerate braided fusion $2$-category $\cB$ \cite{1405.5858,me-TopologicalOrders}. (Topological orders with multiple local vacua are usually rejected by condensed matter theorists because they are ``{unstable}'': a small deformation to the Hamiltonian will condense the system to one of the local vacua.)
 Physically, the objects of $\cB$ are the (1+1)D (quasi)string excitations in the topological order. Write
 $\one \in \cB$ for the invisible identity string, and
  $\Omega\cB = \End_\cB(\one)$ for the fusion $1$-category of its endomorphisms. The objects of $\Omega\cB$ are the (0+1)D (quasi)particle excitations in the topological order. It is a nontrivial fact \cite[Theorem 4]{me-TopologicalOrders} that a topological order (with a single local vacuum) is determined just by its strings and particles: a priori there could have been nontrivial information encoded in the $3$-category of (2+1)D membrane excitations, but in fact all the membranes arise in a systematic way as networks of strings.

The goal of this paper is to classify (3+1)D topological orders with only one simple particle~``$e$'' other than the invisible identity particle $1$. Since $\Omega\cB$ is automatically symmetric monoidal, $e$ necessarily enjoys a $\bZ_2$ fusion rule $e \otimes e \cong 1$: indeed, $\Omega\cB$ is isomorphic either to $\Rep(\bZ_2)$ or $\SVec$. The particle is called ``$e$'' in order to think of it as ``electrically charged,'' but for the gauge group $\bZ_2$ rather than the $\rU(1)$ of Maxwell theory: the world line of an $e$-particle is a Wilson line for the $\bZ_2$ gauge group.
The self-braiding $e \otimes e \to e \otimes e$ is $+1$ (times the identity on $e \otimes e \cong 1$) in $\Rep(\bZ_2)$, and $-1$ in $\SVec$: in other words, $e$ is an \define{emergent boson} if $\Omega\cB \cong \Rep(\bZ_2)$ and an \define{emergent fermion} if $\Omega\cB \cong \SVec$.  
I will prove:

\begin{mainthm}\hypertarget{mainstatement}{}
  There is a unique (up to non-unique isomorphism) (3+1)D topological order with a single nontrivial particle which is an emergent boson (i.e.\ $\Omega\cB \cong \Vec[\bZ_2]$). There are exactly two topological orders with a single nontrivial particle which is an emergent fermion ($\Omega\cB \cong \SVec$).
\end{mainthm}

To give names for ease of future reference, I will call these three topological orders ``$\cR$,'' ``$\cS$,'' and ``$\cT$.'' I will give detailed descriptions of all three 2-categories in Section~\ref{sec.3examples}. Briefly, the three topological orders are:
\begin{itemize}
  \item $\cR$, the unique topological order in which $e$ is a boson, is  plain $\bZ_2$ gauge theory. Another name for it is the \define{3+1D Toric Code}, whose lattice realization was introduced in~\cite{cond-mat/0411752} based on the 2+1D version from~\cite{Kitaev-QCM}, although the continuum gauge-theory model was studied as early as~\cite{MR289087}; see~\cite{MR4073072} for a survey. A direct analysis of the braided fusion 2-category $\cR$ is given in~\cite{2009.06564}. Algebraically, $\cR$ can be described as a Drinfel'd centre in various ways: $\cR \cong \cZ_{(1)}(\Sigma\Rep(\bZ_2)) \cong \cZ_{(1)}((\Sigma\Vec)[\bZ_2])$.
  \item $\cS$ deserves the name ``spin-$\bZ_2$ gauge theory'': instead of a dynamical $\bZ_2$ gauge field, its dynamical field is a spin structure $\eta$, i.e.\ a $\bZ_2$-valued 1-form solving $\d \eta = w_2$, where $w_2$ is some fixed cocycle representative of the second Stiefel--Whitney class (for instance given by a lattice realization of spacetime as in \cite{1505.05856}). As with plain gauge theory, the action functional is trivial. Algebraically, $\cS \cong \cZ_{(1)}(\Sigma\SVec)$.
  \item $\cT$ is an anomalous version of spin-$\bZ_2$ gauge theory studied in \cite[Section~5]{MR3321301} and in
  \cite[Section III.J.4]{KLWZZ2}, and further analyzed in
   unpublished work of Hao Zheng. It can be constructed in the continuum by positing that its only dynamical field is a spin structure $\eta$, but its action functional is $(-1)^{\int \eta w_3} \equiv (-1)^{\int w_2 \Sq^1 \eta}$ (where ``$\equiv$'' means modulo boundary terms). The non-closedness of $\eta$ means that this action has a gravitational anomaly equal to $(-1)^{w_2w_3}$. Algebraically, $\cS$ and $\cT$ are very similar, with identical fusion rules and almost identical braidings. Indeed, the only difference is that in $\cS$, the $\bZ_2$ 1-form symmetry generated by the magnetic string is nonanomalous, whereas in $\cT$  this symmetry has a nontrivial 't Hooft anomaly.
\end{itemize}

In the above bullet points and later, I have adopted the following notation, explained more fully in \S\ref{subsec.f2c}--\ref{subsect.attempt}. If $\cA$ is a fusion $n$-category, then $\cZ_{(1)}(\cA)$ denotes its \define{Drinfel'd centre}: roughly speaking, it consists of elements $X \in \cA$ together with an isomorphism between left and right multiplication by~$X$.
 If $\cB$ is a braided fusion $n$-category, then $\cZ_{(2)}(\cB)$ denotes its \define{M\"uger centre}: roughly speaking, it consists of elements $X \in \cB$ together with an isomorphism between braiding over and under by~$X$. The symbol $\Sigma$ denotes a \define{suspension} operation on monoidal Karoubi-complete higher categories developed in \cite{GJFcond,me-TopologicalOrders}: it takes (braided) fusion $n$-categories to (fusion) $(n{+}1)$-categories, and is essentially the same as what is called ``$\mathbf{Mod}$'' in \cite{2006.08022}.

\begin{remark}\label{remark.LKW}
The first sentence of the \hyperlink{mainstatement}{Main Theorem} follows immediately from the classification of 3+1D topological orders without emergent fermions from~\cite{PhysRevX.8.021074} together with the group cohomology calculation $\H^4_{\mathrm{gp}}(\bZ_2;\bC^\times) = \H^4(\bZ_2[1]; \bC^\times) = 0$. (Here and throughout, given an abelian group $A$, I will write $A[n]$ for the space more typically called $K(A,n)$ or $\rB^nA$.) 
It is worth emphasizing that the arguments in~\cite{PhysRevX.8.021074} and in its sequel~\cite{PhysRevX.9.021005} were given before a complete mathematical definition of ``fusion 2-category'' was first introduced in \cite{Reutter2018}. (Indeed, the preprints of \cite{PhysRevX.8.021074,PhysRevX.9.021005} appeared on the arXiv in 2017 and early 2018, respectively, whereas \cite{Reutter2018} appeared in very late 2018.) Nevertheless, modulo a few small holes filled in~\cite{me-TopologicalOrders}, the argument in \cite{PhysRevX.8.021074} is completely correct, and I will summarize it in Remark~\ref{LKWargument}.

The sequel \cite{PhysRevX.9.021005}, however, contains a small but important mathematical error. Namely, at one point in their argument, the authors of \cite{PhysRevX.9.021005} assume without proof that there is only one (3+1)D topological order $\cB$ with $\Omega\cB \cong \cat{SVec}$; in fact, as stated in the \hyperlink{mainstatement}{Main Theorem}, there are two. The error is quite understandable: as explained in \S\ref{subsec.fer}, the two topological orders with $\cS$ and $\cT$ are extremely similar, and one must look very closely in order to tell them apart. Furthermore, \cite{PhysRevX.9.021005} never claims itself to be mathematically rigorous, but rather to present a physically reasonable argument. Indeed, \cite{PhysRevX.9.021005} is an attempt to classify ``nonanomalous'' 3+1D topological orders, without selecting a mathematically precise definition of the term ``nonanomalous,'' and not an attempt to classify all nondegenerate braided fusion 2-categories. Although \cite{PhysRevX.9.021005} did overlook the 2-category $\cT$, the \hyperlink{mainstatement}{Main Theorem} provides post hoc support for their classification because, as indicated above, $\cT$ arises physically from an anomalous 3+1D system, and probably cannot be realized as a nonanomalous system (compare Remark~\ref{remark.MME}).
\end{remark}

\begin{remark}\label{remark.MME}
  Recall that a braided fusion 1-category $\cC$ is \define{slightly degenerate} if its M\"uger centre $\cZ_{(2)}(\cC)$ is equivalent to $\SVec$. It is a longstanding conjecture that every slightly degenerate braided fusion 1-category admits a \define{minimal modular extension}: an index-2 inclusion into a nondegenerate braided fusion 1-category \cite{MR1990929}.  
  According to an observation of Dmitri Nikshych reproduced in \cite[Remark~V.2]{me-TopologicalOrders}, this conjecture is equivalent to the conjecture that all bosonic 3+1D topological orders are Morita trivial.
  
  Indeed, an even sharper statement (Remark~\ref{rem.MMEWitt}) is available, and serves as one of the motivations of \cite{2006.08022}; I will sketch a proof in \S\ref{subsec.autos}. If $\cC$ is slightly degenerate, then $\cZ_{(1)}(\Sigma\cC)$ will be a 3+1D topological order with line operators $\SVec$. Choices of minimal modular extension of $\cC$, if they exist, are in bijection with choices of isomorphism $\cZ_{(1)}(\Sigma\cC) \cong \cS$.
  
  Thus, good control over the topological orders $\cB$ with $\Omega \cB \cong \SVec$ is quite valuable. In particular, in upcoming work, David Reutter and I will use the detailed descriptions from this paper to show that $\cT$ is not equivalent to $\cZ_{(1)}(\Sigma\cC)$ if $\cC$ is pseudounitary, establishing the minimal modularity conjecture in the pseudounitary case.
\end{remark}

For these reasons, the focus of this paper will be on the second sentence of the \hyperlink{mainstatement}{Main Theorem}. I will in fact give three proofs of the second sentence; these proofs occupy Section~\ref{sec.proofs}. The three proofs vary from hands-on to abstract, and illustrate different features of the classification problem. Before giving the proofs, Section~\ref{sec.3examples} describes in detail the three braided fusion 2-categories $\cR$, $\cS$, and $\cT$. And before that, I will complete this introductory section with some background on the main definitions.

\subsection{Fusion 2-categories} \label{subsec.f2c}

Fusion $2$-categories, first  axiomatized in~\cite{Reutter2018},  are a special type of monoidal $2$-category.
Good references on monoidal and braided monoidal $2$-categories include \cite[Appendix~C]{Schommer-Pries:thesis} and~~\cite[\S2.2]{2006.08022}.
All of my ``$n$-categories'' are ``weak $n$-categories'': composition of (${<}n$)-morphisms is associative only up to higher coherence, but $n$-morphisms form sets and so have strict composition. When $n=2$, ``$2$-category'' is equivalent to ``bicategory.'' 
I will write ``isomorphic'' for what many people might term ``equivalent'': there is no stricter meaningful notion of isomorphism in higher categories.

The extension to higher categories was first axiomatized in \cite{me-TopologicalOrders}, and a very nice treatment is available in~\cite[Section 3]{2011.02859}. Much of  \cite{me-TopologicalOrders} was inspired by the earlier informal treatment of fusion $n$-categories in \cite{1405.5858,KWZ1,KWZ2}.

The definition of ``fusion $2$-category'' from \cite{Reutter2018} parallels the case of fusion $1$-categories: 
\begin{definition}\label{defn.f2c}
  A \define{fusion $2$-category} $\cA$ is a monoidal $2$-category such that:
\begin{itemize}
  \item The underlying $2$-category of $\cA$ is $\bC$-linear finite semisimple.
  \item The unit object $\one \in \cA$  is simple.
  \item All objects  have monoidal duals.
\end{itemize}
\end{definition}

Definition~\ref{defn.f2c} requires some unpacking, especially concerning the notion of ``$\bC$-linear finite semisimple $2$-category.''
An $n$-category $\cA$ is \define{$\bC$-linear} if all sets of $n$-morphisms are $\bC$-vector spaces, and all composition maps are multilinear. When $n=2$, this means that for any objects $X,Y\in \cA$, $\hom_\cA(X,Y)$ should be a $\bC$-linear $1$-category. A $\bC$-linear $2$-category $\cA$ is \define{locally finite semisimple} if $\hom_\cA(X,Y)$ is always a finite semisimple $\bC$-linear category: it should be additively and Karoubi complete and completely reducible, with finitely many simple objects. 

It is easy to define direct sums, and hence additive-completeness, for $n$-categories. 
Let us say that an object is \define{indecomposable} if it does not decompose as a direct sum.
\cite{Reutter2018} also defines a notion of ``simple'' object in a 2-category, which, as in the 1-categorical case, is in general stronger but equivalent to indecomposability if the ambient 2-category is semisimple. As such, I will use the terms ``indecomposable'' and ``simple'' interchangeably.

Higher categorical Karoubi completion was developed for $2$-categories in \cite{Reutter2018} and for general $n$ in \cite{GJFcond}. I will not review the details but I will mention a notation: I will write ``$X \condense Y$'' when there is a \define{condensation} of $X$ onto $Y$; condensation is the $n$-categorical version of split surjection. 
\begin{definition}\label{defn.semisimple2cat}
  A $2$-category $\cA$ is \define{finite semisimple} over $\bC$ if it is locally finite semisimple, all $1$-morphisms have adjoints, it is additive and Karoubi complete, and it has only finitely many isomorphism classes of indecomposable objects.
\end{definition}

The following fact is used implicitly in \cite{Reutter2018}, and proved in $n$-categorical generality in \cite[Lemma~IV.1]{me-TopologicalOrders} and in \cite[Proposition~3.11]{2011.02859}:

\begin{lemma}
  Let $\cA$ be a finite semisimple $2$-category over $\bC$, and $X,Y \in \cA$ a pair of indecomposable objects. Then any nonzero $1$-morphism $X \to Y$ extends to a condensation $X \condense Y$. \qed
\end{lemma}

The composition of condensations $X \condense Y \condense Z$ is a condensation $X \condense Z$, and condensations are nonzero. This implies the \define{Categorical Schur's Lemma} of \cite{Reutter2018}, which says that ``related by a nonzero morphism'' is an equivalence relation on the set of indecomposable objects in $\cA$. Following \cite[Definition 1.2.22]{Reutter2018}, I will say that two indecomposable objects related by a nonzero morphism are in the same \define{component}. More generally:

\begin{definition}
  The \define{component} of an indecomposable object $X \in \cA$ is the full subcategory of $\cA$ on the objects $Y$, not necessarily indecomposable, that admit a condensation $X \condense Y$.
\end{definition}

For example, assuming that $\cA$ is fusion so that the unit object $\one\in \cA$ is indecomposable, it makes sense to talk about its \define{identity component}.
The components of $\cA$ are precisely the indecomposable direct summands of $\cA$ itself \cite[Remark 1.2.23]{Reutter2018}. The set of components is called $\pi_0 \cA$. I will say that a fusion 2-category $\cA$ is \define{connected} if $\pi_0 \cA = \{\one\}$, i.e.\ if the identity component is all of $\cA$.

\begin{remark}\label{remark.baddefn}
  Definition~\ref{defn.semisimple2cat} is best only over an algebraically closed field of characteristic $0$.
  If the characteristic is not zero, subtleties can arise from inseparable fusion $1$-categories; these do not exist in characteristic zero by \cite[Theorem 13]{DSPS}.
  If the field is not algebraically closed, then the Karoubi-completion requirement will typically be incompatible with the requirement that there be only finitely many isomorphism classes of simple objects, because Karoubi-completion produces simple objects for each Morita equivalence class of simple algebra. Rather, the correct finiteness condition for ``finite semisimple $2$-category'' is to ask that the set $\pi_0 \cA$ of components, and not the set of equivalence classes of indecomposables, should be finite.
\end{remark}

A typical example of a finite semisimple $2$-category is the category of finite semisimple modules for a fusion $1$-category $\cF$. I will call this $2$-category $\Sigma\cF$. This follows the notation of \cite{GJFcond}. Specifically, suppose $\cC$ is an additive and Karoubi complete $\bC$-linear $n$-category equipped with a choice of object $\one \in \cC$; then $\Omega \cC := \End_\cC(\one)$ is a monoidal additive and Karoubi complete $\bC$-linear $(n{-}1)$-category. The adjoint to $\Omega$ is $\Sigma$, which takes in a monoidal $(n{-}1)$-category, deloops it to a one-object $n$-category, and Karoubi completes. The unit $\id \to \Omega \Sigma$ of the adjunction $\Sigma \dashv \Omega$, when evaluated on a  monoidal Karoubi-complete $(n{-}1)$-category $\cA$, is an equivalence $\cA \isom \Omega\Sigma\cA$. The counit $\Sigma\Omega \to \id$, when evaluated on a pointed $n$-category $\cC$, is fully faithful, and its essential image is the full subcategory of $\cC$ on those objects $Y \in \cC$ which receive a condensation $\one \condense Y$. From this it follows that $\Sigma$ and $\Omega$ restrict to an equivalence:
$$ \{\text{connected semisimple $n$-categories with indecomposable unit}\} \simeq \{\text{fusion $(n{-}1)$-categories}\}.$$

The $\Omega$ and $\Sigma$ constructions furthermore have good monoidality relations. For instance, if $\cA$ is braided monoidal, then $\Sigma\cA$ is monoidal; if $\cA$ is symmetric monoidal, then so is $\Sigma\cA$. 
Of particular interest is the $(n{+}1)$-category $\Sigma^n\Vec$, and its subgroupoid of invertible objects (and invertible morphisms) $(\Sigma^n\Vec)^\times$.
The fusion $n$-category $\Sigma^{n-1}\Vec$ is called simply ``$n\Vec$'' in \cite{2011.02859}, and in~\S\ref{subsec.proof3} I will write $\W^n$ for its invertible subgroupoid $\W^n := (n\Vec)^\times = (\Sigma^{n-1}\Vec)^\times$.

\subsection{(3+1)D topological orders}\label{subsect.attempt}

The following definition, mentioned in the introduction, is formalized in \cite{me-TopologicalOrders}, based on ideas from \cite{1405.5858}:

\begin{definition}\label{defn.3dtoporder}
  A \define{(3+1)D topological order with a single local vacuum} is a braided fusion $2$-category $\cB$ which is \define{nondegenerate} in the sense of having trivial M\"uger centre.
\end{definition}

Physically, the objects of $\cB$ are the stringlike excitations in the topological order. The worldsheet of a string is a 2-dimensional surface, and so strings are also called \define{surface operators}; the 2-dimensionality is why they are objects of a 2-category. Strings can fuse, and the fusion is braided because in (3+1)D there are two transversal dimensions in which strings can move around. This is why the strings in any (3+1)D topological order form a braided multifusion 2-category $\cB$.

The ``fusion'' condition, demanding that the unit object $\one\in \cB$ be indecomposable, is equivalent to the requirement that the local vacuum be unique. In general, if there were multiple local vacua, then there would be vertex (i.e.\ 0-dimensional) operators that project a state onto a single vacuum; conversely, the set of local vacua is the spectrum of the commutative algebra $\Omega^2\cB$ of vertex operators. This algebra being $\bC$ is equivalent to the unit object being indecomposable.

The final condition in Definition~\ref{defn.3dtoporder} to explain is the nondegeneracy of the braiding. As in the case of fusion $1$-categories, any fusion $2$-category $\cA$ has a \define{Drinfel'd centre} $\cZ_{(1)}(\cA)$, which measures the failure of objects in $\cA$ to be detectable by commuting them with other objects. Its objects are objects $X \in \cA$ together with \define{centrality data}, i.e.\ a natural-in-$Y$ isomorphism $X \otimes Y \cong Y \otimes X$ and further data imposing that this isomorphism is multiplicative in $Y$. The shortest definition is to define $\cZ_{(1)}(\cA)$ as the $2$-category of endomorphisms of $\cA$-as-an-$\cA$-bimodule. By the same token, any braided fusion $2$-category $\cB$ has a \define{M\"uger}, aka \define{sylleptic}, \define{centre} $\cZ_{(2)}(\cB)$. It measures the failure of objects in $\cB$ to be detectable by braiding with other objects. Explicitly, an object of $\cZ_{(2)}(\cB)$ is an object $X \in \cB$ and a natural-in-$Y$ trivialization of the full braiding $\beta_{Y,X} \circ \beta_{X,Y} : X \otimes Y \to Y \otimes X \to X\otimes Y$; this trivialization must furthermore satisfy some braided monoidality requirements. The cleanest definition is to define $\cZ_{(2)}(\cB)$ as the $2$-category of endomorphisms of $\cB$ as a ``braided module'' over itself \cite{2006.08022}.

The nondegeneracy requirement in Definition~\ref{defn.3dtoporder} is that $\cZ_{(2)}(\cB)$ should be the trivial: simply the fusion $2$-category $2\cat{Vec}_\bC = \Sigma\cat{Vec}_\bC$. This axiom was first stated in \cite{1405.5858}, based on ideas in \cite{PhysRevX.3.021009}, and was first given a complete mathematical formulation in \cite{me-TopologicalOrders}. Physically, the axiom encodes the principle of \define{remote detectability}: it says that a nontrivial excitation should be {detectable} by braiding with other excitations.

\begin{remark}
A priori, to specify a (3+1)D topological order one would also need to specify the 3-category of membrane excitations. That this data is not required is a nontrivial theorem about fusion $n$-categories with trivial Drinfel'd centre: they are always connected \cite[Theorem~4]{me-TopologicalOrders}. To see that it is nontrivial, note that it fails if the ``fusion'' condition is dropped: a (3+1)D topological order with multiple local vacua is not determined by its braided multifusion $2$-category of strings; one needs to include information about the membranes, aka \define{domain walls}, that separate the different vacua, and the excitations supported by those membranes (see~\cite{2004.11395} for a physics discussion).
\end{remark}

How might one detect a particle excitation, i.e.\ an object of the symmetric fusion 1-category $\Omega\cB$? 
The easiest way would be to braid it around a string $X \in \cB$. Indeed, $\Omega\cB$ is symmetric monoidal, and so particles certainly cannot detect each other, but the movie of a particle moving around a string would evaluate to a number (or matrix), and ``detection'' occurs when this number is not $1$. Conversely, the easiest way to detect a string is by braiding it against  particles.

In the special case when the simple particles in $\Omega\cB$ form an abelian group $A$ (which occurs when the fusion of any two simple particles is again simple), this leads to  the expectation that the indecomposable strings should form the Pontryagin dual group $\hat{A} = \hom(A,\rU(1))$. A more precise expectation recognizes that the identity component $\Sigma\Omega\cB \subset \cB$ is symmetric monoidal (since $\Omega\cB$ is) and so should be quotiented out, giving the expectation:
\begin{equation}\label{expectation1}
 \pi_0 \cB \cong \hat{A}.
\end{equation}
In the general case with nonabelian particles, Tannakian duality extracts from the symmetric fusion category $\Omega\cB$ a finite group $G$ (isomorphic to $\hat{A}$ in the abelian case), and the reader might expect:
\begin{equation}\label{expectation2}
 \pi_0 \cB \cong G/G := \{\text{conjugacy classes in }G\}.
\end{equation}

Expectations (\ref{expectation1}--\ref{expectation2}) are stated fairly explicitly in \cite{PhysRevX.8.021074,PhysRevX.9.021005}, and form a main ingredient in their arguments. There are, however, reasons to hesitate. Particles in a (3+1)D topological order cannot detect each other, but a particle could be detected by a membrane operator rather than a string. Strings can moreover detect each other: indeed, there are typically nontrivial strings in the identity component of a (3+1)D topological order, and those strings are \emph{not} detected by particles but only by other strings. Thus the pairing between the set of strings and the set of particles is not \emph{a priori} perfect.

In spite of these reasons to hesitate, it turns out that expectations (\ref{expectation1}--\ref{expectation2}) do in fact hold. I do not know a proof of this in general that does not require the full classification \cite{PhysRevX.8.021074}, proved in \cite[Theorem 8 and Corollary V.5]{me-TopologicalOrders} without assuming (\ref{expectation1}--\ref{expectation2}). 

When $\Omega\cB = \Rep(\bZ_2)$ or $\SVec$, as is of interest in this paper, $G = \hat{A} = \bZ_2$. The strings in the non-identity component are called \define{magnetically charged}. It turns out that in all three cases $\cR,\cS,\cT$, there exists a magnetic string $m$ satisfying the fusion rules $m^2 \cong \one$. This is not automatic: for instance, as explained in \S\ref{subsec.bos},
in $\cR$, there is a noninvertible simple magnetic string $m'$, and so it is certainly not true that a random choice of simple magnetic string works; as explained in \S\ref{subsec.fer}, in $\cS$ and $\cT$ there are two non-isomorphic indecomposable magnetic strings, each squaring to the vacuum, and no canonical way to pick one. Given a fusion 2-category $\cC$, it is an interesting question to ask whether the map $\{$indecomposable objects of $\cC\} \to \pi_0 \cC$ splits in any nice way. The explicit calculations of show that for $\cC = \cR,\cS,\cT$, this map does split, but not canonically.

\section{The three examples} \label{sec.3examples}

This section describes in detail the three topological orders $\cR$, $\cS$, and $\cT$. Conditional to the expectations that $\pi_0 = \bZ_2$ in all cases and that in all cases there exists an invertible magnetic string, the calculations in this section constitute a proof of the \hyperlink{mainstatement}{Main Theorem}. Indeed, the first proof in Section~\ref{sec.proofs} will complete the argument by showing directly that if $\cB$ is a 3+1D topological order with $\Omega\cB \cong \SVec$, then $\pi_0\cB = \bZ_2$ and there is an invertible magnetic string. (I  was unable to find a proof of the corresponding statement in the ``emergent boson'' case that did not end up using  the full proof of the classification from \cite{PhysRevX.8.021074}.)

\subsection{The unique topological order in which \texorpdfstring{$e$}{e} is a boson} \label{subsec.bos}

I will now describe in some detail the unique 3+1D topological order $\cR$ with line operators $\Omega\cR \cong \Vec[\bZ_2] \cong \Rep(\bZ_2)$.
As mentioned already in 
Remark~\ref{remark.LKW}, the main result of \cite{PhysRevX.8.021074} implies that $\cR$ is necessarily the 2-category of operators in plain $\bZ_2$ Dijkgraaf--Witten theory, also called the 3+1D Toric Code. The braided fusion categories for general 3+1D Dijkgraaf--Witten theories are studied in \cite{1905.04644}, and the special case of the 3+1D Toric Code is further studied in \cite{2009.06564}; my description is consistent with theirs.

The identity component of $\cR$ is $\Sigma\Omega\cR = \Sigma\Rep(\bZ_2)$. This 2-category is described for example in \cite[Example 1.4.14]{Reutter2018}. It contains two indecomposable objects, corresponding to the two Morita equivalence classes of simple algebra objects in $\Rep(\bZ_2)$. 
One of these objects is the vacuum, corresponding to the trivial algebra $\bC \in \Rep(\bZ_2)$. 
I will call it $\one$ to distinguish it from the vacuum line ``$1$''.

The other object doesn't have a standard name, but corresponds to the algebra object $\cO(\bZ_2) = \bC[x]/(x^2=1) \in \Rep(\bZ_2)$, where the generator $x$ is nontrivially charged under the $\bZ_2$-action. Note that $\cO(\bZ_2)$ is isomorphic to $\bC \oplus \bC$ if the $\bZ_2$-symmetry is ignored, but it is simple as an algebra object in $\Rep(\bZ_2)$: it does not contain any proper $\bZ_2$-equivariant ideals. The corresponding quasistring was first described in \cite{1702.02148} under the name ``Cheshire string''; for that reason, I will call it plain ``$c$.''  I will write $e \in \Omega\cB = \Rep(\bZ_2)$ for the sign representation. The fusion rules are, of course, $e^2 \cong 1$ and
\begin{equation}\label{fusionrule-cheshire}
  c^2 \cong c \oplus c.
\end{equation}
Indeed, equation (\ref{fusionrule-cheshire}) corresponds to an isomorphism of $\bZ_2$-equivariant algebras
$$ \cO(\bZ_2) \otimes \cO(\bZ_2) = \bC[x,y]/(x^2=y^2=1) \cong \cO(\bZ_2) \oplus \cO(\bZ_2)$$
where the two projections correspond to setting $x=y$ or $x = -y$.

\begin{remark}
  The fact that $c$ and $\one$ are in the same component means that there is a non-invertible 1-morphism between them. I say ``between them'' because each 1-morphism has an adjoint, and in a fusion $2$-category left- and right-adjoints for 1-morphisms are always isomorphic. In fact, there is a ``unique'' 1-morphism between $\one$ and $c$ in the sense that the category $\hom_{\Sigma\Rep(\bZ_2)}(\one,c)$ is equivalent to $\Vec$. Note that $\End_{\Sigma\Rep(\bZ_2)}(\one) = \Omega\Sigma\Rep(\bZ_2) = \Rep(\bZ_2)$ by construction, and one can calculate that $\End_{\Sigma\Rep(\bZ_2)}(c)$ is also equivalent to $\Rep(\bZ_2)$.
\end{remark}

\begin{remark}\label{drinfeldcentres}
  Any choice of gapped $n$D boundary condition for any $n$+1D topological order~$\cB$ realizes $\cB$ as the Drinfel'd centre of the fusion $(n{-}1)$-category of operators on the boundary (provided both the bulk and the boundary have unique local vacua). Every Dijkgraaf--Witten theory has a canonical \define{Dirichlet} boundary condition. In 3+1D, the Dirichlet boundary operators form the fusion 2-category $(\Sigma\Vec)^\alpha[G]$, where $\alpha \in \H^4(G[1]; \bC^\times)$ is physically the Dijkgraaf--Witten action, and mathematically the associator data for $(\Sigma\Vec)^\alpha[G]$, which is the 2-category with $G$ many components, each of which is a copy of $\Sigma\Vec$ (hence the parenthesis to distinguish it from $\Sigma(\Vec[G])$). This is why \cite{1905.04644} calculates the Drinfeld centre $\cZ((\Sigma\Vec)^\alpha[G])$. (My ``$\Sigma\Vec$'' is their ``$2\Vec$.'') In particular, the 2-category under consideration arises as
  $$ \cR \cong \cZ_{(1)}((\Sigma\Vec)[\bZ_2]).$$
  
  When the Dijkgraaf--Witten action is trivialized, there is a second canonical boundary condition, namely the \define{Neumann} boundary. It is characterized by the property that the bulk 't Hooft operators condense on the boundary, and the Wilson lines do not. Said mathematically, the Neumann boundary operators form connected the fusion 2-category $\Sigma\Rep(G)$. In particular, the 2-category under consideration arises as
  $$ \cR \cong \cZ_{(1)}(\Sigma\Rep(\bZ_2)).$$
  This description has an analog for the 3+1D topological order $\cS$; see Remark~\ref{remark.spinNeumann}.
\end{remark}

From either description in Remark~\ref{drinfeldcentres}, it is easy to calculate that $\cR$ has a single non-identity component: in the notation of \cite[Remark 1.2.23]{Reutter2018}, $\pi_0 \cR = \bZ_2$. The objects in the nontrivial component are the \define{magnetically charged} strings. The characterizing feature of magnetic charge is that a magnetic string links nontrivially with the electron $e$: when an $e$ particle is moved all the way around a magnetic string, the state picks up a phase factor of $-1$.

It turns out that the magnetic component has the same structure as the identity component. There is an indecomposable object $m \in \cB$ which is invertible: in fact, $m^2 \cong \one$. This implies already that $\End_\cB(m) \cong \Rep(\bZ_2)$. The other indecomposable object is $m' := m \otimes c$. It is not invertible, but rather satisfies the fusion rules
$ (m')^2 \cong c\oplus c.$
As with the identity component, there is a nonzero, but non-invertible, 1-morphism connecting $m$ and $m'$: $\hom_\cC(m,m') \cong \Vec$.
This is enough to determine $\cR$ as a linear $2$-category: it has two components, each equivalent to $\Sigma\Rep(\bZ_2)$.

To really understand a braided fusion $2$-category  requires more than the fusion rules: the braided monoidal structure is data. In general, braided monoidality data can be quite complicated. This is familiar from the fusion 1-category, where the associator data, for example, consists of matrices solving some overdetermined inhomogeneous polynomial equations, modulo certain changes of basis. Also familiar from the 1-category case is the fact that these data simplify dramatically for ``grouplike'' categories in which all the simple objects are invertible. Indeed, fusion categories $\cC$ in which the simple objects have fusion rules following a finite group $G$ are classified by ordinary group cohomology $\H^3_{\mathrm{gp}}(G;\bC^\times) = \H^3(G[1]; \bC^\times)$; if $G$ is abelian, then the braided fusion categories $\cC$ with fusion rules $G$ are classified by $\H^4(G[2]; \bC^\times)$. Specifically, these classes are the Postnikov invariants for extensions $\bC^\times[2].G[1]$ and $\bC^\times[3].G[2]$, respectively. These extensions arise as (the classifying spaces of) the invertible subcategories of $\cC$, i.e.\ the subcategories of invertible objects and invertible 1-morphisms. The full category can be reconstructed from its invertible subcategory by promoting the group $\bC^\times$ of morphisms back to the $\bC$-algebra $\bC$, and then additively and Karoubi-completing. I will refer to this process of promoting $\bC^\times \leadsto \bC$ and then additively and Karoubi-completing just as \define{linearization}.

Exactly the same thing happens in fusion 2-categories. In general the data is complicated, but it is merely cohomological when the 2-category arises as a linearization of a homotopy type with $\bC^\times$ in top degree. This is the situation for $\cR$. Indeed, let $\cR^\times$ denote its invertible braided fusion sub-2-category. The objects of $\cR^\times$ are just $\one$ and $m$, and the 1-morphisms are just $\Aut(\one) \cong \Aut(m) \cong \{1,e\}$. The 2-morphisms are just $\Aut(1) \cong \Aut(e) \cong \bC^\times$. Said another way, we have:
\begin{equation}
  \cR^\times = \bC^\times[2].\bZ_2[1].\bZ_2[0].
\end{equation}
In fact, ignoring the monoidal structure and treating $\cB^\times$ just as a 2-category, these extensions all split. Indeed, the only possible Postnikov extension datum is the connecting map for $\bC^\times[2].\bZ_2[1]$, which lives in $\H^3(\bZ_2[1]; \bC^\times)$ and encodes the associator (but not braiding) on $\Omega\cR = \Rep(\bZ_2)$. This associator is trivial.
Furthermore:
\begin{lemma}\label{lemma.Rlinearization}
  $\cR$ is the linearization of $\cR^\times$.
\end{lemma}
\begin{proof}
  The Lemma essentially follows from the earlier discussion about $\Sigma\Rep(\bZ_2)$. To understand the details, it helps to work in layers. The top two layers of $\cR^\times$ are $\bC^\times[2].\bZ_2[1]$, which linearize to the 1-category $\Rep(\bZ_2)$. To produce a 2-category requires further Karoubi completion, and this is how the Cheshire string $c$ is created. The bottom layer $\bZ_2[0]$ means that there is a second equivalent component.
\end{proof}

It follows that to describe the braided monoidal structure on $\cB$, it suffices to describe the braided monoidal structure on $\cR^\times$. This is in turn described by its double classifying space $\rB^2\cR^\times$, which has the same homotopy groups shifted by 2:
\begin{equation}\label{extensionR}
  \rB^2\cR^\times =  \bC^\times[4].\bZ_2[3].\bZ_2[2].
\end{equation}

There are many spaces with these homotopy groups. To find the right one, it helps to first study the quotient extension $\bZ_2[3].\bZ_2[2] = \{1,e\}[3].\{\one,m\}[2]$, which is classified by a class $\beta \in \H^4(\bZ_2[2]; \bZ_2) \cong \bZ_2$. In terms of  braided category theory, this class records the self-braiding
$$ \beta_{m,m} : m\otimes m \to m\otimes m.$$
Specifically, $\beta$ is trivial or not depending on whether this self-braiding is $1$ or $e$. Using either description from Remark~\ref{drinfeldcentres}, it is not hard to show that $\beta_{m,m}$ is trivial, and so the sub-extension $\bZ_2[3].\bZ_2[2]$ is isomorphic to a product $\bZ_2[3] \times \bZ_2[2]$.

To complete the description (\ref{extensionR}) of $\cR$, it therefore suffices to give a class
\begin{equation}\label{equation.alpha}
 \alpha \in \H^5(\bZ_2[3] \times \bZ_2[2]; \bC^\times) = \H^5(\{1,e\}[3] \times \{\one,m\}[2];\bC^\times).
\end{equation}
Using the K\"unneth and Hurewicz theorems, one finds:
\begin{equation}\label{Kunneth}
 \H^5(\bZ_2[3] \times \bZ_2[2]; \bC^\times) \cong \underbrace{\H^5(\bZ_2[3]; \bC^\times)}_{\cong \bZ_2} \times \underbrace{\H^3(\bZ_2[3];\bC^\times) \otimes \H^2(\bZ_2[2];\bC^\times)}_{\cong \bZ_2\otimes\bZ_2 \cong \bZ_2} \times \underbrace{\H^5(\bZ_2[2];\bC^\times)}_{\cong \bZ_2}.
\end{equation}
(Hurewicz implies that the ``extension'' classes from K\"unneth vanish in this degree.)

It is useful to give names to the generators on the right-hand side of (\ref{Kunneth}). Recall that the $\bZ_2$-cohomology ring $\H^\bullet(\bZ_2[n]; \bZ_2)$ is freely generated over the Steenrod algebra by a single degree-$n$ generator. I will write ``$E$'' for the degree-$3$ generator of $\H^\bullet(\{1,e\}[3]; \bZ_2)$, and ``$M$'' for the degree-$2$ generator of  $\H^\bullet(\{1,m\}[2]; \bZ_2)$. Then for example the $\bZ_2$-cohomology of $\{1,e\}[3] \times \{1,m\}[2]$ is generated over the Steenrod algebra by $E$ and $M$.
There is a map $x \mapsto (-1)^x$ from $\bZ_2$-cohomology to $\bC^\times$-cohomology. In general, not every $\bC^\times$-cohomology class is in the image of this map: for example, $\H^4(\bZ_2[2]; \bC^\times) \cong \bZ_4$, and so any map from $\bZ_2$-cohomology cannot by surjective. But it turns out that the map $\H^5(\{1,e\}[3] \times \{1,m\}[2]; \bZ_2) \to \H^5(\{1,e\}[3] \times \{1,m\}[2]; \bC^\times)$ is surjective (although not injective). 
Specifically, the three $\bZ_2$s on the right-hand side of (\ref{Kunneth}) are generated by the classes
\begin{equation}\label{classnames}
(-1)^{\Sq^2 E}, \quad (-1)^{EM}, \quad \text{and} \quad (-1)^{\Sq^2 \Sq^1 M} = (-1)^{M \Sq^1 M}.
\end{equation}

What product $\alpha(E,M)$ of these classes describes the braiding on $\cR^\times$ (and hence on $\cR$)? The restriction $\alpha|_{M=0}$ to $\{1,e\}[3]$ describes the braiding on $\Omega\cR = \Rep(\bZ_2)$. The statement that $e$ is a boson is the statement that this braiding is trivial. Thus $\alpha|_{M=0} = +1$, and the term $(-1)^{\Sq^2 E}$ does not appear in $\alpha(E,M)$.
On the other hand, the term $(-1)^{EM}$ must appear in $\alpha$, since it is what encodes that the $e$-particle and the $m$-string detect each other.
What about the third term $(-1)^{\Sq^2 \Sq^1 M} = (-1)^{M \Sq^1 M}$? The answer is that it doesn't matter:
\begin{lemma}\label{lemma.autoboson}
  With notation as above, the classes $(-1)^{EM}$ and $(-1)^{EM + M \Sq^1 M}$ lead to equivalent braided fusion $2$-categories.
\end{lemma}
\begin{proof}
  A straightforward calculation shows that $\pi_0 \Aut(\bZ_2[3] \times \bZ_2[2]) \cong \bZ_2$. The nontrivial automorphism can be understood in coordinates as
  $$ (E,M) \mapsto (E',M') = (E + \Sq^1 M, M).$$
  In words, this is the automorphism of $\bZ_2[3] \times \bZ_2[2]$ which restricts to the identity on $\bZ_2[3]$ and projects to the identity on $\bZ_2[2]$, but uses the nontrivial class in $\H^3(\bZ_2[2]; \bZ_2)$ to choose a different splitting.
  Pulling back $(-1)^{EM}$ along this automorphism gives
$ 
(-1)^{E'M'} = (-1)^{(E + \Sq^1 M)M}
$, and so this automorphism exchanges the two classes.
\end{proof}

\begin{remark}\label{remark.bos.SS}
Lemma~\ref{lemma.autoboson} could be understood as ``experimental confirmation'' of the uniqueness of the 3+1D topological order $\cR$ with $\Omega\cR = \Rep(\bZ_2)$. As explained 
in Remark~\ref{remark.LKW}, this uniqueness follows directly from the general classification given in \cite{PhysRevX.8.021074}.

  Suppose that this general classification was unavailable, but that there was a direct proof that $\pi_0\cR = \bZ_2$ and that $\cR$ contained an invertible magnetic string $m$. Then $\cR$ would arise as an extension of shape ``$\Sigma\Rep(\bZ_2).\bZ_2^F$,'' where the notation is supposed to indicate that this is not a central extension, but rather $\pi_0\cR = \{\one,m\}$ acts on the identity component as the 1-form symmetry $(-1)^F$. After remembering that $m$ does indeed source a 1-form symmetry, one sees that extensions of this shape are the same as extensions of shape $(\bC^\times[4] \times \bZ_2[3]).\bZ_2^F[2]$, where $\bC^\times[4] \times \bZ_2[3]$ arises as a classifying space of the braided monoidal 2-groupoid $\Sigma\Rep(\bZ_2)^\times$. The symmetric monoidal structure on $\Rep(\bZ_2)$ makes it natural to think of $\bC^\times[4] \times \bZ_2[3]$ as the degree-4 term in the spectrum for the cohomology theory $\H^\bullet(-;\bC^\times \times \bZ_2[-1]) := \H^\bullet(-;\bC^\times) \times \H^{\bullet - 1}(-;\bZ_2)$. Extensions of shape $(\bC^\times[4] \times \bZ_2[3]).\bZ_2^F[2]$ are then classified by a \define{twisted} cohomology group $\H^5(\bZ_2^F[2]; \bC^\times \times \bZ_2[-1])$, where the twisting encodes that $\bZ_2^F$ is supposed to act nontrivially on $\bC^\times[4] \times \bZ_2[3]$ and is recorded in the notation by the $f$ superscript.
  
  There is an (easy!)\ Atiyah--Hirzebruch spectral sequence computing this twisted cohomology. The $E_2$ page has two rows:
  $$\left\{\begin{matrix} E_2^{\bullet,1} = \H^\bullet(\bZ_2[2];\bZ_2) \\ E_2^{\bullet,0} = \H^\bullet(\bZ_2[2];\bC^\times) \end{matrix}\right\} \Rightarrow \H^\bullet(\bZ_2^F[2]; \bC^\times \times \bZ_2[-1]).$$
  As above, write $M$ for the degree-2 generator of $\H^\bullet(\bZ_2[2];\bZ_2)$. Then the only differential in this spectral sequence is
  $$ \d_2 : \H^\bullet(\bZ_2[2];\bZ_2) \to \H^{\bullet+2}(\bZ_2[2];\bC^\times), \quad X \mapsto (-1)^{XM}.$$
  The unique nonzero class $(-1)^{M\Sq^1 M} = (-1)^{\Sq^2 \Sq^1 M} \in \H^5(\bZ_2[2]; \bC^\times)$ is in the image of this differential, and the unique nonzero class $M^2 = \Sq^2 M \in \H^4(\bZ_2[2]; \bZ_2)$ is not in the kernel of this differential, and so:
  $$ \H^5(\bZ_2^F[2]; \bC^\times \times \bZ_2[-1]) = 0.$$
  This is another way to package the earlier calculations, and establishes directly the uniqueness of $\cR$ subject to the existence of an invertible magnetic string.
\end{remark}

\subsection{The two topological orders in which \texorpdfstring{$e$}{e} is a fermion} \label{subsec.fer}

I now turn to the case of topological orders $\cB$ with $\Omega\cB \cong \SVec$. As in \S\ref{subsec.bos}, the first step is to understand the identity component $\Sigma\SVec$. Since $\SVec$ and $\Rep(\bZ_2)$ are monoidally equivalent, $\Sigma\SVec$ and $\Sigma\Rep(\bZ_2)$ are equivalent as semisimple 2-categories. In particular, $\Sigma\SVec$ contains two simple objects (up to equivalence): the identity object $\one$ and a Cheshire string $c$.

The first difference between the emergent-boson and emergent-fermion cases shows up in the fusion rules for the Cheshire string. In the bosonic case, the Cheshire string satisfied $c^2 \cong c \oplus c$~(\ref{fusionrule-cheshire}). When $e$ is a fermion, however, the Cheshire string corresponds to the simple associative superalgebra $\Cliff(1) \in \SVec$. The tensor product of algebra objects depends on the braiding in the ambient category. As superalgebras, there is a well-known isomorphism~\cite{MR167498}
$$ \Cliff(1) \otimes \Cliff(1) = \Cliff(2) \cong \operatorname{Mat}(1|1) \simeq \bC.$$
Here $\operatorname{Mat}(1|1)$ means the matrix algebra of endomorphisms of $\bC^{1|1} \in \SVec$, and $\simeq$ denotes a super Morita equivalence. As such, in $\Sigma\SVec$ the Cheshire string satisfies the fusion rule:
\begin{equation} \label{fusionrule-cheshire-fermionic}
  c^2 \cong \one.
\end{equation}

Since $c$ is invertible, its self-braiding $\beta_{c,c} : c\otimes c \to c\otimes c$ is either the identity or is (the identity times) the fermion $e$.  The latter holds in fact holds:
\begin{equation}\label{self-braiding-c}
  \beta_{c,c} \cong e : c \otimes c \to c\otimes c
\end{equation}
 The self-braiding of the fermion $\beta_{e,e}$ is also nontrivial. In other words, the braided monoidal groupoid $\Sigma\SVec^\times$ corresponds to a space 
$$ \rB^2\Sigma\SVec^\times = \bC^\times[4].\bZ_2[3].\bZ_2[2] = \bC^\times[4].\{1,e\}[3].\{\one,c\}[2]$$
with both Postnikov k-invariants nontrivial:
$$ \Sq^2 : \{\one,c\} \to \{1,e\}, \qquad (-1)^{\Sq^2} : \{1,e\} \to \bC^\times.$$
These Postnikov invariants are stable --- just given by Steenrod operations --- because $\Sigma\SVec$ is  symmetric monoidal and not just braided monoidal. 
Indeed, as a symmetric monoidal object, $\Sigma\SVec^\times$ corresponds to the ``extended supercohomology'' $\SH^\bullet$ of \cite{WangGu2017}. (My degree conventions are that $\SH^n(-)$ classifies maps into $\bC^\times[n].\bZ_2[n-1].\bZ_2[n-1]$.)
A priori to determine a three-level extension like this requires more than just the k-invariants, but in the $\SH^\bullet$ case the two possible extensions with these k-invariants are isomorphic \cite[\S5.4]{MR3978827}.

To extend $\Sigma\SVec$ to a topological order requires adding in a magnetic component --- for now, I~will take for granted the expectation (\ref{expectation1}) that $\pi_0\cB = \bZ_2$,  so there is a single magnetic component. Choose a simple magnetic string $m$ in this component. 
 By \cite[Theorem~B]{JFYu}, any choice $m$ is invertible, and the other simple magnetic string is $m' = m\otimes c$. Since $m$ is invertible, its self-braiding $\beta_{m,m} : m\otimes m \to m\otimes m$ is given by an invertible particle $x \in \{1,e\}$. The self-braiding on $m^2$ is then given by $x^4$. In particular, equation (\ref{self-braiding-c}) is inconsistent with $m^2 \cong c$, and so:
 \begin{equation} \label{fusionrule-m-fermionic}
  m^2 \cong \one.
  \end{equation}

\begin{remark}
  Unlike the case considered in Lemma~\ref{lemma.Rlinearization}, $\Sigma\SVec$ is \emph{not} the linearization of its invertible subgroupoid $\Sigma\SVec^\times$. Rather, $\Sigma\SVec$ is already the linearization of just its sub-groupoid $\bC^\times[2].\{1,e\}[1]$, since $\SVec$ is the linearization of $\bC^\times[1].\{1,e\}[0]$. 
  If one linearized all of $\Sigma\SVec^\times$, one would get a 2-category with two components.
  Similarly, the full topological order $\cB$ is not the linearization of its invertible subgroupoid $\cB^\times$, but just of $\bC^\times[2].\{1,e\}[1].\{\one,m\}[0]$.
\end{remark}

Thus to classify the topological orders $\cB$ with $\Omega\cB \cong \SVec$, one must simply work out the possible extensions $\bC^\times[4].\{1,e\}[3].\{\one,m\}[2]$ consistent with the braiding requirements. The first question to work out is the quotient extension $\{1,e\}[3].\{\one,m\}[2]$, which classifies the self-braiding $\beta_{m,m} : m \otimes m \to m\otimes m$ up to isomorphism. In Proposition~\ref{prop.magbos}, I will prove that the self-braiding of any magnetic string is necessarily trivial, and so $\{1,e\}[3].\{\one,m\}[2]$ is (noncanonically) a product. Then, as in equation (\ref{equation.alpha}) and the surrounding discussion, to determine the topological order it suffices to give a class
$$\alpha(E,M) \in \H^5(\{1,e\}[3] \times \{\one,m\}[2]; \bC^\times) = \{(\pm 1)^{\Sq^2 E}\} \times \{(\pm 1)^{EM}\} \times \{(\pm 1)^{\Sq^2 \Sq^1 M}\}.$$

The statement that $e$ is a fermion means that $\alpha|_{M=0}$ should be nontrivial, 
and so the term $(- 1)^{\Sq^2 E}$ is present. The term $(-1)^{EM}$ is also present. As in the case of $\cR$,  the nontrivial coefficient on $EM$ is what enforces remote detectability: it says that an $e$-particle and an $m$-string link nontrivially. This leaves two choices for $\alpha$:
\begin{align}
  \cS &\quad\leftrightarrow\quad \alpha(E,M) = (-1)^{\Sq^2 E + EM}, \label{S-braiding} \\
  \cT &\quad\leftrightarrow\quad \alpha(E,M) = (-1)^{\Sq^2 E + EM + \Sq^2 \Sq^1 M}. \label{T-braiding}
\end{align}

It remains to prove that, unlike the case considered in Lemma~\ref{lemma.autoboson}, these two choices really do lead to inequivalent topological orders. First, one could wonder about the automorphism $(E,M) \mapsto (E',M') = (E + \Sq^1 M, M)$ used in that Lemma, which corresponded to choosing a different splitting of the extension $\{1,e\}.\{\one,m\}$. Pulling back the class in (\ref{S-braiding}) along this automorphism, one calculates:
$$ (-1)^{\Sq^2 E' + E'M'} = (-1)^{\Sq^2(E + \Sq^1 M) + (E + \Sq^1 M)M} = (-1)^{\Sq^2 E + \Sq^2 \Sq^1 M + EM + M \Sq^1 M}$$
by the linearity of $\Sq^2$. But $(-1)^{M \Sq^1 M} = (-1)^{\Sq^2 \Sq^1 M}$ as classes in $\H^5(\bZ_2[2]; \bC^\times)$, and so one ends up back to $(-1)^{\Sq^2 E + EM}$. The same calculation shows that class in (\ref{T-braiding}) is also fixed by $(E,M) \mapsto (E',M') = (E + \Sq^1 M, M)$.

The other possible worry is whether the choice of magnetic string matters: perhaps $\cS$ and $\cT$ are exchanged by switching $m \mapsto m' = m \otimes c$. 
In fact, this does not happen. Indeed, I will show in Proposition~\ref{prop.mswitch} that $\cS$ has an automorphism that switches $m$ with $m'$, and so $\cS$ and $\cT$ cannot be exchanged by such a switch (since the previous paragraph already establishes that $\cS$ and $\cT$ are not exchanged by any automorphism that fixes $m$). Thus $\cT$ also has an automorphism switching $m$ with $m'$, and the two topological orders truly are inequivalent.

\begin{remark}\label{remark.spinNeumann}
The topological order $\cS$ is achieved physically by replacing the $\bZ_2$ gauge field in \S\ref{subsec.bos} with a spin structure, to produce the TFT that counts spin structures. 
As explained in Remark~\ref{drinfeldcentres}, any boundary condition for a topological order realizes that topological order as a Drinfel'd centre. There is no sensible Dirichlet boundary condition for spin-$\bZ_2$ gauge theory because there is no ``trivial'' spin structure and so no sensible meaning to the sentence ``the spin-structure trivializes at the boundary.'' However, there is a Neumann boundary condition for spin-$\bZ_2$ gauge theory in which the spin structure is arbitrary on the boundary. 
This boundary condition corresponds to realizing $\cS$ as
$$ \cS \cong \cZ_{(1)}(\Sigma\SVec).$$

The topological order $\cT$ is achieved by the anomalous spin-$\bZ_2$ gauge theory studied in \cite[Section~5]{MR3321301} and \cite[Section III.J.4]{KLWZZ2}. I do not expect that it has a description as a Drinfel'd centre; compare Remarks~\ref{remark.MME} and \ref{remark.MMEredux}.
\end{remark}

\begin{remark}\label{inthewild}
  Suppose you meet a topological order $\cB$ in the wild, and check that $\Omega\cB \cong \SVec$. How do you tell whether $\cB \cong \cS$ or $\cT$? The answer can be read off from the extension $\alpha(E,M)$ in (\ref{S-braiding}--\ref{T-braiding}) as follows. First, choose a magnetic string $m$. Second, choose an isomorphism $m^2 \cong 1$: this is the choice of splitting $\{1,e\}.\{\one,m\} \cong \{1,e\}\times\{\one,m\}$. These two choices together allow you to find inside $\cB$ a braided sub-$2$-category whose only objects are $\{\one,m\}$, whose only 1-morphisms are identities, and whose only 2-morphisms are scalars. The braided monoidal structure on this sub-$2$-category is then described by the class $\alpha|_{E=0} \in \H^5(\bZ_2[2]; \bC^\times) $. If this class is trivial, then $\cB \cong \cS$, whereas if it is nontrivial, then $\cB \cong \cT$.
  
  Incidentally, the nontrivial choice $(-1)^{\Sq^2 \Sq^1 M} \in \H^5(\bZ_2[2]; \bC^\times) \cong \bZ_2$ is a stable class, meaning that this braided monoidal sub-2-category is in fact symmetric monoidal (in a unique way, since the desuspension map $\H^{k+3}(\bZ_2[k];\bC^\times) \to \H^5(\bZ_2[2];\bC^\times)$ is an isomorphism for $k\geq 2$).
\end{remark}

\begin{remark}\label{remark.monoidalST}
  Forgetting the braiding on $\cS$ or $\cT$, and remembering only the monoidal structure, corresponds to working with the extension $\bC^\times[3].(\{1,e\}[2] \times \{\one,m\}[1])$ classified by the desuspension of the class $\alpha(E,M)$ from (\ref{S-braiding}--\ref{T-braiding}). It is easy to desuspend stable classes like $(-1)^{\Sq^2 E}$ and $(-1)^{\Sq^2 \Sq^1 M}$: their formulas stay the same, just with $E$ and $M$ now interpreted as the degree-$2$ and degree-$1$ generators of $\H^\bullet(\{1,e\}[2];\bZ_2)$ and $\H^\bullet(\{\one,m\}[1];\bZ_2)$, respectively. But $\H^4(\bZ_2[1];\bC^\times)$ is trivial, and indeed $(-1)^{\Sq^2 \Sq^1 M} = (-1)^{M^4} = +1$, since $M$ has now degree $1$. The class $(-1)^{EM}$ also desuspends to the trivial class. All together, one finds that $\cS$ and $\cT$ are monoidally, but not braided monoidally, equivalent, and both are furthermore monoidally equivalent to the symmetric monoidal fusion 2-category $(\Sigma\SVec) [\bZ_2]$.
\end{remark}

\begin{remark}\label{remark.fer.SS}
  As in Remark~\ref{remark.bos.SS}, there is a more direct spectral-sequence-based proof of the classification of topological orders $\cB$ with $\Omega\cB \cong \cat{SVec}$, subject only to the assumption that $\pi_0\cB = \bZ_2$. First, by \cite[Theorem~B]{JFYu}, every indecomposable magnetic string is invertible. Thus the question is to classify braided extensions of shape ``$(\Sigma\SVec).\bZ_2^F$,'' where the superscript $f$ indicates that the fibre $\Sigma\SVec$ is not central, but rather transforms nontrivially under the 1-form symmetry group $\bZ_2^F = \pi_0\cB$. Since $(\Sigma\SVec)^\times$ classifies supercohomology $\SH^\bullet$, extensions of this shape are classified by the ``twisted supercohomology'' group $\SH^{5}(\bZ_2^F[2])$.
  
  There is an Atiyah--Hirzebruch spectral sequence computing this twisted cohomology group:
  $$ \H^i(\bZ_2^F[2]; \SH^j(\pt)) \Rightarrow \SH^{i+j}(\bZ_2^F[2]).$$
  Here $\SH^0(\pt) = \bC^\times$, $\SH^1(\pt) = \bZ_2$, $\SH^2(\pt) = \bZ_2$, and all other rows vanish. The twisting does not affect the $E_2$ page (because a 1-form symmetry cannot act on an ordinary group like $\SH^\bullet(\pt)$), but rather appears in the differentials. The $\d_2$ differentials are
  \begin{align*}
     \d_2 & : E_2^{i,2} = \H^i(\bZ_2^F[2]; \bZ_2) \to E_2^{i+2,1} = \H^{i+2}(\bZ_2^F[2]; \bZ_2) && X \mapsto \Sq^2 X + MX \\
     \d_2 & : E_2^{i,1} = \H^i(\bZ_2^F[2]; \bZ_2) \to E_2^{i+2,0} = \H^{i+2}(\bZ_2^F[2]; \bC^\times) && X \mapsto (-1)^{\Sq^2 X + MX}
  \end{align*}
  The $E_2$ page with nontrivial differentials drawn in looks like:
  $$
\begin{tikzpicture}[anchor=base]
\path
(0,0) node {$\bC^\times$} ++(.75,0) node {$0$} ++(.75,0) node (b2) {$\bZ_2$} ++(.75,0) node {$0$} ++(.75,0) node {$\bZ_4$} ++(.75,0) node {$\bZ_2$}++(.75,0) node (e2) {$\cdots$}
(0,.5) node (b1) {$\bZ_2$} ++(.75,0) node {$0$} ++(.75,0) node (a2) {$\bZ_2$} ++(.75,0) node {$\bZ_2$} ++(.75,0) node (e1) {$\bZ_2$} ++(.75,0) node (c2) {$\bZ_2^2$} ++(.75,0) node (d2) {$\cdots$}
(0,1) node (a1) {$\bZ_2$} ++(.75,0) node {$0$} ++(.75,0) node {$\bZ_2$} ++(.75,0) node (c1) {$\bZ_2$} ++(.75,0) node (d1) {$\bZ_2$} ++(.75,0) node {$\cdots$} 
;
\draw[->] (-.75,-.125) -- ++(5.25,0);
\draw[->] (-.4,-.5) -- ++(0,2);
\path  (0,-.5) node {$0$} ++(.75,0) node {$1$} ++(.75,0) node {$2$} ++(.75,0) node {$3$} ++(.75,0) node {$4$} ++(.75,0) node {$5$} ++(.75,0) node {$i$}
(-.75,0) node {$0$} ++(0,.5) node {$1$} ++(0,.5) node {$2$} ++(0,.5) node {$j$}
;
\draw[thick] 
(a1.mid) -- (a2.mid)
(b1.mid) -- (b2.mid)
(c1.mid) -- (c2.mid)
(d1.mid) -- (d2.mid)
(e1.mid) -- (e2.mid)
;
\path (b1) ++ (-2,0) node {$E_2^{i,j} =$};
\end{tikzpicture}
$$
This is a truncation to degrees $j \leq 2$ of the spectral sequence considered in \S\ref{subsec.proof3}. In particular, Lemma~\ref{lemma.nod3} implies that the limit of this spectral sequence in total degree 4 is 
\begin{equation}\label{eqn.16-fold-way}
\SH^4(\bZ_2^F[1]) \cong \bZ_{16}.
\end{equation}
Indeed, this group is identifiable with the kernel of the map $\cW \to \rS\cW$ from the ``Witt group'' of (Morita equivalence classes of) nondegenerate braided fusion 1-categories to the ``super Witt group'' of slightly-degenerate braided fusion 1-categories, which is shown to be a $\bZ_{16}$ in \cite[5.3]{MR3022755}. I will give another description of this isomorphism in Remark~\ref{rem.MMEWitt}.

It follows that all the entries in total degree $4$ survive to the $E_\infty$ page, and so the $d_3$ differential vanishes in this range of degrees. Therefore the $\bZ_2$ in bidegree $(5,0)$ survives to $E_\infty$, and:
\begin{equation} \SH^5(\bZ_2^F[2]) \cong \bZ_2.\end{equation}
Thus there are two topological orders $\cB$ with $\Omega\cB = \SVec$ and $\pi_0\cB = \bZ_2$, and they are distinguished by a class in ordinary cohomology $\H^5(\bZ_2[2];\bC^\times)$, reproducing Remark~\ref{inthewild}.

The fact that the $\bZ_2$ in bidegree $(3,2)$ supports a nontrivial $\d_2$ differential is ``the same'' as the fact (\ref{fusionrule-m-fermionic}) that $m^2 \cong \one$ and not $c$. The fact that the $\bZ_2$ in bidegree $(4,1)$ supports a nontrivial $\d_2$ differential is ``the same'' as the fact that the self-braiding $\beta_{m,m}$ is necessarily $1$, which I will prove directly in Proposition~\ref{prop.magbos}.
\end{remark}

\subsection{Automorphisms of \texorpdfstring{$\cR$}{R}, \texorpdfstring{$\cS$}{S}, and \texorpdfstring{$\cT$}{T}}\label{subsec.autos}

To  fully understand a topological order, it is interesting to work out its symmetry group. The following language has become fairly standard in the physics literature, following \cite{MR3321281}:

\begin{definition}
  The \define{$p$-form symmetries} of a higher-categorical object $\cX$ are $\pi_p \Aut(\cX)$.
\end{definition}

Let $\cA$ be the fusion $n$-category of all operators (including higher-dimensional membrane operators) in an $n$+1D topological order. It is almost true that $\Aut(\cA) = \cA^\times$ is the monoidal sub-3-groupoid of invertible operators. The reason this fails is that of course the centre of $\cA$ acts trivially. The ``remote detectability'' axiom for a general topological order requires that this centre is as trivial as possible:
$$ \cZ_{(1)}(\cA) = \Sigma^{n-1}\Vec = n\Vec.$$
The correct general statement asserts that $\Aut(\cA)$ is the quotient of $\cA^\times$ by $\cZ_{(1)}(\cA)^\times = (\Sigma^{n-1}\Vec)^\times$:
\begin{proposition}\label{prop.topologicalNoether}
  Let $\cA$ be an $n$+1D topological order in the sense of \cite{me-TopologicalOrders}, i.e.\ a multifusion $n$-category with trivial Drinfel'd centre. Then there is a fibre sequence
  $$ \rB \cA^\times \to \rB \Aut(\cA) \overset \kappa \to \rB(\Sigma^n\Vec)^\times. $$
\end{proposition}
This proposition should be thought of as a topological version of Noether's theorem. The spaces $(\Sigma^\bullet\Vec)^\times$ fit together into a loop spectrum: $\Omega((\Sigma^\bullet\Vec)^\times) = (\Sigma^{\bullet-1}\Vec)^\times$. Following \S\ref{subsec.proof3}, I will call (a cohomological shift of) this spectrum by the name ``$\W^\bullet$'':
\begin{equation} \label{eqn.nameW} \W^n := (\Sigma^{n-1}\Vec)^\times = (n\Vec)^\times. \end{equation}
Complete details, and an explanation of the name, can be found in \S\ref{subsec.proof3}. The extension $\cA^\times = (n\Vec)^\times.\Aut(\cA)$ has as its Postnikov class  $\kappa \in \W^{n+2}(\rB \Aut(\cA))$, which should be thought of as the \define{universal 't Hooft anomaly} for $\cA$.
\begin{proof}
  Any automorphism $f \in \Aut(\cA)$ of $\cA$ can be encoded as an invertible $\cA$-$\cA$ bimodule ``$\cA_f$'': it is $\cA$ as a left-module, but the right action is twisted by $f$. To remember $f$, one remembers the ``pointing'' $1_\cA \in \cA_f$. Since $\cA$ is a topological order, it is invertible in the Morita $(n{+}2)$-category of multifusion $n$-categories \cite[Theorem~2]{me-TopologicalOrders}, and so its monoidal $(n{+}1)$-category of (unpointed) bimodules is trivial, i.e.\ equivalent to $\Sigma^n\Vec$.
  In other words, the group map $\Aut(\cA) \to (\Sigma^n\Vec)^\times$ sending $f \mapsto \cA_f$, as an unpointed bimodule, has kernel the choice of pointing. But the choice of pointing is exactly an element of $\cA^\times$.
\end{proof}

In the case at hand, $n = 3$ and $\cA = \Sigma\cB$ for $\cB \in \{\cR,\cS,\cT\}$ --- $\cA$ has trivial Drinfel'd centre because $\cB$ has trivial M\"uger centre \cite[Corollary IV.2]{me-TopologicalOrders}.
In particular, $\Aut(\cA) = \Aut(\cB)$ and $\pi_i \cA^\times = \pi_{i-1} \cB^\times$ for $i \geq 1$.
Unfolding the fibre sequence from Proposition~\ref{prop.topologicalNoether} and applying the notation from (\ref{eqn.nameW}) produces a long exact sequence in homotopy groups:
\begin{multline}\label{nother-LES}
  1 \to \bC^\times \to \pi_2 \cB^\times \to \pi_3 \Aut(\cB) \to \pi_0 \W^1 \to \pi_1 \cB^\times \to \pi_2 \Aut(\cB) \to \pi_0 \W^2 
  \\
  \to \pi_0 \cB^\times \to \pi_1 \Aut(\cB) \to \pi_0 \W^3 \to \pi_0(\Sigma\cB)^\times \to \pi_0 \Aut(\cB) \to \pi_0 \W^4
\end{multline}
Very explicitly, $\pi_p \Aut(\cB)$ are the $p$-form symmetries of $\cB$, whereas $\pi_{p-1}\cB^\times$ are the invertible $(p-1)$-morphisms in $\cB$. It is not true in general that this sequence is exact on the far right: in the proof of Proposition~\ref{prop.topologicalNoether}, not every invertible bimodule needs to be isomorphic to one of the form $\cA_f$ for an automorphism $f$.

A few of the terms in (\ref{nother-LES}) are quick to fill in. Since $\cB$ is fusion, $\pi_2 \cB^\times = \bC$ and the first arrow is an isomorphism. As explained in Lemma~\ref{lemma.Wpt}, $\pi_0 \W^1$, $\pi_0 \W^2$, and $\pi_0 \W^3$ are all trivial.
Inspecting the descriptions of $\cR,\cS,\cT$ in \S\ref{subsec.bos} and \S\ref{subsec.fer} then gives:
\begin{corollary}\label{cor.autos}
  The 2-form symmetry groups of $\cR$, $\cS$, and $\cT$ are
  $$ \pi_2 \Aut(\cR) \cong \pi_2 \Aut(\cS) \cong \pi_2 \Aut(\cT) \cong \bZ_2 = \{1,e\}.$$
  The 1-form symmetry groups are
$$    \pi_1 \Aut(\cR) \cong \bZ_2 = \{\one,m\}, \qquad \pi_1 \Aut(\cS) \cong \pi_1 \Aut(\cT) \cong \bZ_2^2 = \{\one,c,m,m'\},
$$
  where in the latter case $m$ and $m'$ are the two magnetic strings. 
  Moreover, the extension $\pi_2.\pi_1$ is trivializable in the $\cR$-case, whereas in the case of $\cS$ and $\cT$ it is classified by the nontrivial braiding $\beta_{c,c} = e$.
  \qed
\end{corollary}

\begin{remark}\label{remark.anomaly2}
  As is clear from their explicit descriptions in \S\ref{subsec.fer},  the two braided fusion 2-categories $\cS$ and $\cT$ are almost isomorphic. They are monoidally isomorphic (see Remark~\ref{remark.monoidalST}), and most of their braiding data agrees: the only difference is a few $\pm1$s in the self-braidings of the magnetic strings. As a result, the automorphism groups $\Aut(\cS)$ and $\Aut(\cT)$ are isomorphic, and their actions on $\cS$ and $\cT$ are ``the same.'' The difference is only in the anomalies: one can summarize Remark~\ref{inthewild} by saying that in $\cS$ the 1-form $\bZ_2$-symmetry generated by $m$ is nonanomalous, whereas in $\cT$ is has a nontrivial 't Hooft anomaly.
\end{remark}

This leaves the zero-form symmetries. The triviality of $\pi_0\W^3$ gives a left-exact sequence
$$ 1 \to (\Sigma\cB)^\times \to \pi_0 \Aut(\cB) \to \pi_0\W^4.$$
It is worth explaining this sequence physically. The right-most term  is nothing but the Witt group $\cW$ of nondegenerate braided fusion categories studied in \cite{MR3039775,MR3022755}. The middle term is of course the group of 0-form automorphisms of $\cB$. Any 0-form automorphism can be realized as a gapped topological 2+1D $\cB$-$\cB$ interface aka membrane operator. The objects of~$\Sigma\cB$, on the other hand, are the membrane operators $X$ that can be built from condensing networks of surface operators, i.e.\ for which there is a condensation $\mathrm{Vac} \condense X$, where $\mathrm{Vac}$ denotes the trivial (aka invisible or vacuum) 2+1D operator. These are precisely the gapped topological 2+1D membrane operators that admit a gapped topological 1+1D boundary condition. 
Other invertible 2+1D membrane operators may only admit gapless boundary conditions. The map $\pi_0 \Aut(\cB) \to  \cW$ records the obstruction to gappability of the boundary by recording the equivalence class, modulo gapped 0+1D boundary interfaces, of any conformal boundary condition.

As this discussion illustrates, it will be hard to analyze $\pi_0\Aut(\cB)$ by directly computing $\Sigma\cB$. 
Fortunately, these groups are accessible by inspecting the spectral sequences in Remarks~\ref{remark.bos.SS} and~\ref{remark.fer.SS}. Indeed, it is a general feature of such cohomological computations that if the choices are parameterized by degree-$n$ cohomology, then the automorphisms (up to higher isomorphism) of a given choice are parameterized by degree-$(n{-}1)$ cohomology. As such, there are isomorphisms
  \begin{gather}
   \pi_0 \Aut(\cR) \cong \H^4(\bZ_2^F[2]; \bC^\times \times \bZ_2[-1]) \cong \bZ_2, \label{autR} \\
   \pi_0 \Aut(\cS) \cong \pi_0 \Aut(\cT) \cong \SH^4(\bZ_2^F[2]) \cong \bZ_{16}. \label{autS}
  \end{gather}

To finish my description of these topological orders, I will explain the groups $\pi_0\Aut(\cR)$ and $\pi_0\Aut(\cS)$ in detail. 

The former is very easy. By Lemma~\ref{lemma.Rlinearization}, an automorphism of $\cR$ is precisely an automorphism of $\bZ_2[3] \times \bZ_2[2] = \{1,e\}[3] \times \{\one,m\}[2]$ preserving $(-1)^{EM}$. As explained in the proof of Lemma~\ref{lemma.autoboson}, $\bZ_2[3] \times \bZ_2[2]$ has only one nontrivial automorphism, and it does not preserve $(-1)^{EM}$. This is not a contradiction, however:  one must preserve $(-1)^{EM}$ not just up to abstract isomorphism, but up to specified isomorphism. This
 requires extra data, valued in a torsor for cohomology of one degree lower. And indeed $\H^3(\bZ_2[3] \times \bZ_2[2]; \bC^\times)  \cong \bZ_2$, confirming the isomorphism (\ref{autR}).

To understand $\Aut(\cS)$, it will be convenient to work with a pair of closely-related symmetric monoidal 5-categories, which are built following the general constructions in \cite[Section 8]{JFS}. First is the \define{pointed Morita category} $\cat{Alg}_2(2\cat{KarCat}_\bC)$ with:
  \begin{description}
     \item[0-morphisms] Braided monoidal Karoubi-complete $\bC$-linear 2-categories.
     \item[1-morphisms] Monoidal Karoubi-complete $\bC$-linear 2-category bimodules.
     \item[2-morphisms] Pointed Karoubi-complete $\bC$-linear 2-category bimodules.
     \item[3-morphisms] Pointed functors of 2-category bimodules.
     \item[4-morphisms] Pointed natural transformations of bimodule functors.
     \item[5-morphisms] Pointed modifications of natural transformations.
  \end{description}
  Details of what it means for a braided fusion 2-category to act on a fusion 2-category may be found in \cite{2006.08022}.
  By ``pointed,'' I mean that the bimodules come equipped with distinguished objects, which are to be preserved by higher morphisms. The composition of 1- and 2-morphisms is a balanced tensor product. In order to match with order of composition, I will consider a bimodule $_\cA \cM _\cB$ to be a morphism $\cB \to \cA$.
  
Second is the \define{unpointed Morita category} $\cat{Mor}_2(2\cat{KarCat}_\bC)$ with:
  \begin{description}
     \item[0-morphisms] Braided monoidal Karoubi-complete $\bC$-linear 2-categories.
     \item[1-morphisms] Monoidal Karoubi-complete $\bC$-linear 2-category bimodules.
     \item[2-morphisms] Karoubi-complete $\bC$-linear 2-category bimodules.
     \item[3-morphisms] Functors of 2-category bimodules.
     \item[4-morphisms] Natural transformations of bimodule functors.
     \item[5-morphisms] Modifications of natural transformations.
  \end{description}
Given a braided monoidal Karoubi-complete $\bC$-linear 2-category $\cB$, I will write $[\cB]$ for the corresponding object in $\cat{Mor}_2(2\cat{KarCat}_\bC)$.
There is obviously a forgetful functor $\cat{Alg}_2(2\cat{KarCat}_\bC) \to \cat{Mor}_2(2\cat{KarCat}_\bC)$ which forgets the pointing data. 

Although seemingly very similar, these categories end up having quite different flavour:
\begin{proposition}
  Two braided monoidal Karoubi-complete $\bC$-linear 2-categories are equivalent in $\cat{Alg}_2(2\cat{KarCat}_\bC)$ if and only if they are braided-monoidally equivalent; automorphism groups computed in $\cat{Alg}_2(2\cat{KarCat}_\bC)$ agree with automorphism groups in the braided-monoidal sense. 
  Every object (resp.\ 1-morphism) in $\cat{Alg}_2(2\cat{KarCat}_\bC)$ is 2- (resp.\ 1-) dualizable, but
  the only 5-dualizable object in $\cat{Alg}_2(2\cat{KarCat}_\bC)$ is the identity object $2\Vec$.
  
  Equivalences in $\cat{Mor}_2(2\cat{KarCat}_\bC)$ are Morita equivalences. An object $[\cB] \in \cat{Mor}_2(2\cat{KarCat}_\bC)$ is 5-dualizable if and only if $\cB$ is fusion, and $[\cB]$ is invertible if and only if $\cB$ is a topological order.
\end{proposition}
\begin{proof}
  The first sentence follows from \cite[Proposition~3.2.45]{ScheimbauerThesis}. The second sentence is essentially to \cite[Theorem~6.1]{GwSch18}; see also \cite[Theorem~7.3]{me-Heisenberg}. The third sentence is essentially the definition of ``Morita equivalence.'' The last sentence is \cite[Theorems~1 and~2 and Corollary~IV.2]{me-TopologicalOrders}; a similar result in a slightly different context is \cite[Theorem~1.1]{BJSS2020}. 
\end{proof}

Now fix an isomorphism $N : \cS \isom \cZ_{(1)}(\Sigma\SVec)$; as explained in Remark~\ref{remark.spinNeumann}, such an isomorphism arises by realizing $\cS$ as spin-$\bZ_2$ gauge theory and choosing the Neumann boundary condition, hence the name $N$ for the isomorphism. This is equivalent to choosing a way for $\cS$ to act (from the right, say) on the fusion 2-category $\Sigma\SVec$. In other words, it realizes $\cN := \Sigma\SVec_\cS$ as a 1-morphism $\cN : \cS \to 2\Vec$ in $\cat{Alg}_2(2\cat{KarCat}_\bC)$. Furthermore, the corresponding 1-morphism $[\cN] : [\cS] \to [2\Vec]$ in $\cat{Mor}_2(2\cat{KarCat}_\bC)$ is an equivalence because the original map $N$ was an isomorphism. $\cZ_{(1)}(\Sigma\SVec)$ also acts on $\Sigma\SVec$ from the left; this produces the dual 1-morphism $\cN^* : 2\Vec \to \cS$ in $\cat{Alg}_2(2\cat{KarCat}_\bC)$ and its image $[\cN^*] = [\cN]^{-1} : [2\Vec] \to [\cS]$ in $\cat{Mor}_2(2\cat{KarCat}_\bC)$.

For any $f \in \Aut(\cS)$, write $\cS_f$ for the corresponding bimodule $\cS$-$\cS$ bimodule: as a right $\cS$-module, $\cS_f = \cS$, whereas as a left $\cS$-module, the action is twisted by $f$. 
Then $N \circ f$ is another isomorphism of braided 2-categories $\cS \isom \cZ_{(1)}(\Sigma\SVec)$, i.e.\ another boundary condition. It corresponds to the 1-morphism $\cN \circ \cS_f = \cN \otimes_\cS \cS_f : \cS \to 2\Vec$ in $\cat{Alg}_2(2\cat{KarCat}_\bC)$. Physically, $\cS_f$ is an invertible membrane operator in the 3+1D ``bulk'' $\cS$, and $\cN \circ \cS_f$ is the boundary condition built by bringing this membrane to the 2+1D boundary $\cN$.

Consider now a ``sandwich'' or ``slab,'' drawn in the left-hand part of Figure~\ref{figure.slab}, in which the interior is the 3+1D topological order $\cS$ with left-hand boundary $\cN \circ \cS_f$ and right-hand boundary the original Neumann right-boundary $\cN^*$. In terms of the 5-category  $\cat{Alg}_2(2\cat{KarCat}_\bC)$, this slab is precisely the composition
\begin{equation}\label{eqn.slab} \cA := \cN \circ \cS_f \circ \cN^* = \Sigma\SVec \underset{\cS}\otimes \cS_f \underset{\cS}\otimes \Sigma\SVec\end{equation}
Note that the image $[\cA] : [2\Vec] \to [2\Vec]$ in $\cat{Mor}_2(2\cat{KarCat}_\bC)$ is invertible, since it is the composition of $\cat{Mor}$-invertible 1-morphisms $[\cN] \circ [\cS_f] \circ [\cN]^{-1}$. Thus, $\cA$ is a 2+1D topological order in the sense of \cite{me-TopologicalOrders}: it is a multifusion 2-category with trivial centre.

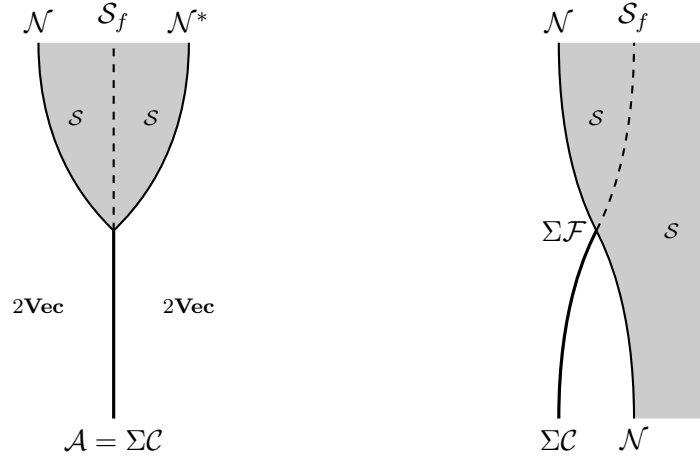
\begin{figure}[t]
  \centering
\begin{tikzpicture}[scale=.5]
  \fill[black!20!white] (0,0) .. controls +(-.5,.5) and +(0,-3) .. (-2,5) -- (2,5) .. controls +(0,-3) and +(.5,.5) .. (0,0);
  \draw[thick] (0,0) .. controls +(-.5,.5) and +(0,-3) .. (-2,5) node[anchor=south] {$\cN$};
  \draw[thick] (0,0) .. controls +(.5,.5) and +(0,-3) .. (+2,5) node[anchor=south] {$\cN^*$};
  \draw[thick, dashed] (0,0) -- (0,5) node[anchor=south] {$\cS_f$};
  \draw[very thick] (0,0) -- (0,-5) node[anchor=north] {$\cA = \Sigma\cC$};
  \path (-1,3) node {$\scriptstyle \cS$} (1,3) node {$\scriptstyle \cS$} (-2,-2) node {$\scriptstyle 2\Vec$} (2,-2) node {$\scriptstyle 2\Vec$};
\end{tikzpicture}
\hspace{1.5in}
\begin{tikzpicture}[scale=.5]
\fill[black!20!white] (1,-5) .. controls +(0,3) and +(.25,-.5) .. (0,0) .. controls +(-.25,.5) and +(0,-3) .. (-1,5) -- (3,5) -- (3,-5) -- cycle;
\draw[thick] (0,0) .. controls +(-.25,.5) and +(0,-3) .. (-1,5) node[anchor=south] {$\cN$};
\draw[thick,dashed] (0,0) .. controls +(+.25,.5) and +(0,-3) .. (1,5) node[anchor=south] {$\cS_f$};
\draw[very thick] (0,0) .. controls +(-.25,-.5) and +(0,3) .. (-1,-5) node[anchor=north] {$\Sigma\cC$};
\draw[thick] (0,0) .. controls +(.25,-.5) and +(0,3) .. (1,-5) node[anchor=north] {$\cN$};
\path (0,3) node {$\scriptstyle \cS$} (2,0) node {$\scriptstyle \cS$};
\path (0,0) node[anchor=east] {$\Sigma\cF$};
\end{tikzpicture}
\caption{Left: a graphical depiction of the definition (\ref{eqn.slab}) of the 2+1D topological order $\cA$ as a ``slab'' built from the Neumann boundary condition $\cN : \cS \to 2\Vec$ and the automorphism $f \in \Aut(\cS)$. Right: the same figure, rotated a bit, recovers the boundary condition $\cN \circ \cS_f$ from a Morita equivalence $\cF : \cC \otimes \SVec \simeq \SVec$.}
\label{figure.slab}
\end{figure} 

As with any ``sandwich'' or ``slab'' construction, the $k$-dimensional operators in $\cA$ are realized as triples consisting of: a $(k{+}1)$-dimensional operator in the bulk $\cS$, a way for this operator to end on the left-boundary, and a way for this operator to end on the right boundary. For example, to give a vertex operator in $\cA$ requires giving a line in $\cS$ which can end on both boundaries. But the worldline of the fermion~$e$ cannot end on the Neumann boundary --- it does not ``condense'' on the boundary --- and so the only vertex operators in $\cA$ are ends of the vacuum line. Thus $\cA$ is in fact fusion. By \cite[Theorem~5]{me-TopologicalOrders}, the lines in $\cA$ form a nondegenerate braided fusion 1-category $\cC := \Omega\cA$, and $\cA$ is recovered as its suspension $\cA \cong \Sigma\cC$.

Thus to understand $\cA$ it suffices to understand its category $\cC$ of line operators. Every line operator is realized as a surface operator in $\cS$ together with ways for that surface to end on the boundaries. The trivial surface operator $\one \in \cS$ provides two lines to $\cA$: the vacuum identity line~$1$ and a fermion~$e$. In other words, $\cC  \supset \SVec$. The Cheshire string $c \in \cS$ does not condense on either boundary, and may be ignored. More precisely, the Cheshire string can end in the bulk in a noninvertible way, and this ending may be moved to the boundary; the resulting object in $\cC$ is $1\oplus e$.

\begin{remark}\label{idonidcomp}
  Each boundary condition $\cN,\cN^*$ provides an inclusion $\SVec \subset \cC$, since those boundary conditions are each ways that $\cS$ can act on $\Sigma\SVec$ --- the copies of $\SVec$ arise as the left- and right- endpoints of the vacuum string in $\cS$. Because the lines can be lifted through the bulk, these two subcategories $\SVec \subset \cC$ are in fact equal as subcategories. If $\SVec$ admitted 0-form automorphisms, then the two inclusions $\SVec \mono \cC$ might differ by an automorphism of $\SVec$, but this cannot happen since $\pi_0 \Aut(\SVec)$ is trivial. However, $\SVec$ does has 1-form automorphisms, which could show up if one were to consider an isomorphism $f \cong f'$ of elements in $\Aut(\cS)$. The cleanest statement is to require $f$ to live in the kernel of $\Aut(\cS) \to \Aut(\Omega\cS) = \Aut(\SVec)$.
\end{remark}

This leaves strings from the magnetic component of $\cS$. These strings are characterized by having nontrivial linking with the fermion $e$. It follows that any line in $\cA$ produced from their endings will also link nontrivially with $e$. In other words:

\begin{proposition} \label{prop.MME1}
  The nondegenerate braided fusion category $\cC = \Omega\cA$ is a \define{minimal modular extension (MME)} of $\SVec = \{1,e\}$: the relative M\"uger centre of the inclusion $\SVec \subset \cC$ is just $\SVec$. \qed
\end{proposition}

It is a famous fact, due to \cite{MR2200691,MR3022755}, that $\SVec$ has exactly $16$ minimal modular extensions, namely the modular tensor categories $\Spin(k)_1$ (which depend only on $k$ mod 16); see \cite[\S4.4]{MR3613518} for details.

The construction $f \mapsto \cC$ can be reversed. Suppose $\cC$ is a minimal modular extension of $\SVec$,  and consider the braided fusion category $\cC \otimes \SVec$. Since the braiding in $\cC$ is nondegenerate, $\cZ_{(1)}(\Sigma\cC) \cong \Sigma\Vec $ is trivial, and so there is a canonical isomorphism of 3+1D topological orders
\begin{equation} \label{eqn.MMEa}
 \cZ_{(1)}(\Sigma(\cC \otimes \SVec)) \cong \cZ_{(1)}(\Sigma\cC \otimes \Sigma\SVec) \cong \cZ_{(1)}(\Sigma\cC) \otimes \cZ_{(1)}(\Sigma\SVec)  \cong \cZ_{(1)}(\Sigma\SVec).
\end{equation}
Physically, this isomorphism corresponds to realizing $\cC \otimes \SVec$ as a boundary condition for $\cS = \cZ_{(1)}(\Sigma\SVec)$ simply by taking the Neumann boundary $\cN$ and stacking on a noninteracting copy of~$\cA = \Sigma\cC$. That it gives a true boundary condition is precisely the statement that the composition $\cA \circ \cN : \cS \to \Sigma\Vec$ in $\cat{Alg}_2(2\cat{KarCat}_\bC)$ lives over an invertible morphism $[\cA] \circ [\cN]$ in $\cat{Mor}_2(2\cat{KarCat}_\bC)$.

Since $\cC$ is a minimal modular extension of $\SVec$, it comes equipped with a distinguished subcategory $\cC \supset \SVec$, and so $\cC \otimes \SVec \supset \SVec \otimes \SVec$. Condense the diagonal in $\SVec \otimes \SVec$, i.e.\ condense the composite boson $e \otimes e \in \cC \otimes \SVec$. The result is yet another copy of $\SVec$, and the condensation procedure  produces a Morita equivalence of braided fusion 1-categories
\begin{equation} \label{eqn.CSVec}
  \cF : \cC \otimes \SVec \isom \SVec.
\end{equation}
The fusion category $\cF$ is a copy of $\cC$, and its $\cC$-action is the standard one.

Since Morita equivalences induce isomorphisms of centres, (\ref{eqn.CSVec}) leads to an isomorphism
\begin{equation}\label{eqn.MMEb}
  \cZ_{(1)}(\Sigma(\cC \otimes \SVec)) \cong \cZ_{(1)}(\Sigma\SVec).
\end{equation}
Composing the isomorphisms (\ref{eqn.MMEa}) and (\ref{eqn.MMEb}) gives an automorphism $f \in \Aut(\cS)$. By its construction, this automorphism trivializes on the identity component $\Sigma\SVec$ of $\cS$.

In the 5-categorical language,  (\ref{eqn.CSVec}) corresponds to a 2-morphism
\begin{equation} \label{eqn.SigmaCSvec}
 \Sigma\cF : \cA \circ \cN \Rightarrow \cN \circ \cS_f
\end{equation}
of 1-morphisms $\cS \to 2\Vec$ in $\cat{Alg}_2(2\cat{KarCat}_\bC)$. That $\cF$ was a Morita equivalence translates to the statement that $[\Sigma\cF] : [\cN] \Rightarrow [\cA] \circ [\cN]$ is invertible in $\cat{Mor}_2(2\cat{KarCat}_\bC)$.

\begin{proposition}
  Theses constructions $f \mapsto \cC$ and $\cC \mapsto f$ are inverse to each other.
\end{proposition}
\begin{proof}
  The 2-morphism $\Sigma\cF$ from (\ref{eqn.SigmaCSvec}) is nothing but the ``mate'' of the isomorphism (\ref{eqn.slab}), i.e.\ its image under the equivalence
  $$ \hom(\cA \Rightarrow \cN \circ \cS_f \circ \cN^*) \cong \hom(\cA \circ \cN \Rightarrow \cN \circ \cS_f).$$
  The only thing to confirm is that, given $\cC$, the corresponding 2-morphism $\Sigma\cF \in \hom(\cA \circ \cN \Rightarrow \cN \circ \cS_f)$ has an isomorphism as its mate. Since $\Sigma\cF$ lives over a Morita equivalence between Morita invertible 1-morphisms, its mate  is certainly a Morita equivalence. By Proposition~\ref{prop.MME1}, it is a Morita equivalence between (suspensions of) minimal modular extensions of $\SVec$. Unpacking definitions shows that the mate of $\Sigma\cF$ is simply $\Sigma\cF$ again, but thought of as a bimodule in a different way. On the other hand, the construction of $\cF$ showed that, as a $\cC$-module, it was the standard one. This implies that the 2-morphism $\cA = \Sigma\cC \Rightarrow \cN \circ \cS_f \circ \cN^*$ in $\cat{Alg}_2(2\cat{KarCat}_\bC)$ is the ``modulation'' of a homomorphism $\cA \to \cN \circ \cS_f \circ \cN^*$ of fusion 2-categories, or equivalently a homomorphism $\cC = \Omega\cA \to \Omega(\cN \circ \cS_f \circ \cN^*)$ of braided fusion 2-categories. But a homomorphism of minimal modular extensions of $\SVec$ is automatically an isomorphism.
\end{proof}

\begin{remark}\label{rem.MMEWitt}
The same construction in fact proves the following much more general statement, first conjectured in \cite[Theorem$^{\mathrm{ph}}$ 3.37]{KLWZZ} (see also \cite[Propositions~33 and~34]{KLWZZ2}). Let $\cC$ be a braided fusion 1-category with M\"uger centre $\cE = \cZ_{(2)}(\cC)$. Then the groupoid of minimal modular extensions of $\cC$ is naturally isomorphic to the groupoid of isomorphisms $\cZ_{(1)}(\Sigma\cC) \cong \cZ_{(1)}(\Sigma\cE)$ whose restriction to $\Sigma\cE$  is the identity. (Given the higher-categorical nature of the problem, there is data involved in identifying the restriction to $\Sigma\cE$ with the identity isomorphism. The necessity of this data is indicated in Remark~\ref{idonidcomp}.)
In particular, if $\cC = \cE$ and if $\cE$ has no 0-form automorphisms (as is the case for $\cE = \SVec$), then $\pi_0 \Aut(\cZ_{(1)}(\Sigma\cE))$ is isomorphic to the group of minimal modular extensions of $\cE$ from~\cite{MR3613518}. 

  For example, for any finite group $G$, the minimal modular extensions of $\Rep(G)$ are classified by $\H^3(\rB G; \bC^\times)$  \cite[Theorem 4.22]{MR3613518}. Taking $G = \bZ_2$ confirms equation (\ref{autR}). For $\cE = \SVec$, this confirms (the $\cS$ part of) equation (\ref{autS}), and also explains the relationship, alluded to in Remark~\ref{remark.fer.SS}, between $\SH^4(\bZ_2^F[2]) \cong \bZ_{16}$ and the kernel of the map $\cW \to \rS\cW$ from the bosonic to fermionic Witt groups.
\end{remark}

The last remaining question is to understand ``computationally'' how an automorphism $f \in \Aut(\cS)$ determines the corresponding minimal modular extension $\cC$ of $\SVec$. As mentioned above, two of the simple objects of $\cC$ --- the vacuum $1$ and the fermion $e$ --- arise as ways that the invisible string $\one \in \cS$ can end, and the other objects come from the magnetic strings $m$ and $m'$. In the Neumann boundary condition $\cN$, one of these --- $m$, say --- condenses, whereas the $m' \mapsto c$, the Cheshire string in $\Sigma \SVec$. Suppose that $m$ also condenses in $\cN \circ \cS_f$. Then the magnetic particles in $\cC$, i.e.\ the simple objects that braid nontrivially with $e$, are invertible, and there are two of them, differing by a copy of $e$. This follows simply because $m \in \cS$ can end in an invertible way on both boundary conditions. On the other hand, suppose that $m'$ condenses in new $\cN \circ \cS_f$ boundary condition, and $m \mapsto c$. The Cheshire string $c$ can end, but only in an noninvertible way, and that ending is fixed under fusing with $e$. As a result,  $\cC$ contains a noninvertible simple object fixed under fusing with $e$.

In other words, the automorphism $f$ exchanges which string in $\cS$ condenses if and only if the corresponding braided fusion category $\cC$ is one of the eight Ising-type categories among the 16 minimal modular extensions of $\SVec$. The existence of such Ising-type categories implies:

\begin{proposition} \label{prop.mswitch}
  $\cS$ (and hence $\cT$) admits an automorphism which exchanges the magnetic strings $m$ and $m'$. \qed
\end{proposition}

\section{Three proofs of the Main Theorem} \label{sec.proofs}

In this section I will prove the \hyperlink{mainstatement}{Main Theorem}. As explained in the introduction, I will focus on the second sentence. The reason is that I do now know any proof of the first sentence that does not simply repeat the argument from \cite{PhysRevX.8.021074}. As for the second sentence, I will give three proofs. The first, in \S\ref{subsec.proof1} is extremely hands-on and direct: I will complete the proof started in \S\ref{subsec.fer} by showing directly that any topological order $\cB$ with $\Omega\cB \cong \SVec$ has only two components $\pi_0 \cB \cong \bZ_2$, thus realizing the expectation (\ref{expectation1}). The problem with hands-on proofs is that they tend to use special features of the problem, and it is often not clear how to apply their ideas to other cases. In the second proof, in \S\ref{subsec.proof2}, is more generalizable: by using a super version of the classification from \cite{PhysRevX.8.021074}, the ``base change'' $\cB \otimes \Sigma\SVec$ of any topological $\cB$ with $\Omega\cB \cong \SVec$ is known, and the \hyperlink{mainstatement}{Main Theorem} reduces to a ``Galois descent'' problem for the ``Galois extension'' $\Sigma\Vec \to \Sigma\SVec$. The last proof, in \S\ref{subsec.proof3}, is the most abstract: I will directly apply the classification statement of \cite[Corollary~V.5]{me-TopologicalOrders}.

\begin{remark}\label{LKWargument}
For completeness, I will very briefly review the argument from \cite{PhysRevX.8.021074}, specialized to the case in the first sentence of the \hyperlink{mainstatement}{Main Theorem}. Suppose that $\cB$ is a 3+1D topological order with $\Omega\cB = \Rep(\bZ_2)$. Then the emergent boson $e$ is condensible. More precisely, the algebra object $\cO(\bZ_2) \in \Rep(\bZ_2)$ is commutative, and this provides a (commutative) algebra structure to the Cheshire string $c \in \Sigma\Rep(\bZ_2)$, which is moreover a condensation monad in the sense of \cite{GJFcond}.  The image of this condensation monad is a new 3+1D topological order $\cB'$ equipped with a Morita equivalence to $\cB$; physically, the 3+1D topological orders $\cB$ and $\cB'$ are separated by a gapped topological 2+1D interface. But since $e$ condenses, $\cB'$ has no line operators, and so is trivial by \cite[Remark~IV.3]{me-TopologicalOrders}, and the interface is a boundary condition for~$\cB$. Thus (by Remark~\ref{drinfeldcentres}) $\cB = \cZ_{(1)}(\cA)$, where $\cA$ is the fusion 2-category of surfaces on this boundary condition. Finally, the whole procedure started with the choice of commutative algebra object $\cO(\bZ_2) \in \Rep(\bZ_2)$, and its automorphism group is $\bZ_2$ (as a 0-form group, with no higher-form symmetries). This means that the boundary condition has 0-form automorphism group $\bZ_2$, implemented by invertible surface operators, and it has no higher-dimensional operators, and so $\cA = (\Sigma\Vec)^\omega[\bZ_2]$ for some class $\omega \in \H^4(\bZ_2[1];\bC^\times) = 0$. (By \cite{JFYu}, there are no non-invertible indecomposable surface operators on the boundary.)
\end{remark}

\subsection{First proof of the Main Theorem: direct analysis of \texorpdfstring{$\pi_0 \cB$}{pi0 B}} \label{subsec.proof1}

I will now give a hands-on direct proof of (the second sentence of) the \hyperlink{mainstatement}{Main Theorem}, building on the calculations in \S\ref{subsec.fer}. Throughout this subsection, I assume that $\cB$ is a (possibly-degenerate) braided fusion $2$-category with $\Omega\cB \cong \SVec$. To start off, note that:

\begin{lemma}[{\cite[Theorem~B]{JFYu}}]
  The indecomposable objects of $\cB$ are all invertible. Each component of $\cB$ has exactly two objects, related by tensoring with the Cheshire string $c \in \Sigma\Omega\cB$. The components also form a group $\pi_0 \cB = \{\text{indecomposable objects}\} / \{\one,c\}$. \qed
\end{lemma}

Thus it suffices to work just with the invertible subgroupoid $\cB^\times \subset \cB$, and I will simply say ``string'' when I mean an object of $\cB^\times$, leaving the word ``indecomposable'' implicit.

Since the fermion $e$ squares to the vacuum, the linking of $e$ with any string $x$ is given just by a sign $\pm 1$. Furthermore, the function
$$ \text{link with $e$} : \{\text{strings}\} \to \{\pm 1\}$$
is a homomorphism. Since $e$ links trivially with the Cheshire string, this homomorphism factors through the group of components $\pi_0 \cB$.
Call $x$ \define{magnetically charged} if $e$ links nontrivially with $x$, and \define{magnetically neutral} if they link trivially.

The invertibility of a string $x$ furthermore implies that the self-braiding $\beta_{x,x} : x \otimes x \to x \otimes x$ is either the identity or it is (the identity times) the fermion $e$. This is analogous to asking whether an invertible particle is a boson or a fermion. Due to this, I will say say that $x$ is a \define{bosonic string} if $\beta_{x,x} \cong \id_{x \otimes x}$, and a \define{fermionic string} if $\beta_{x,x} \cong e \otimes \id_{x\otimes x}$. Note that this distinction records some but not all of the self-statistics of $x$. For example, as explained in Remark~\ref{inthewild},
in both topological orders~$\cS$ and~$\cT$ the magnetic string $m$ is ``bosonic,'' but the full braiding data is different in the two cases.

The following lemma should remind the reader of the 2+1D Toric Code and its fermionic cousin $\Spin(8)_1$:
\begin{lemma}
  If $x$ is magnetically neutral, then one of $x$ and $x \otimes c$ is bosonic and the other is fermionic. If $x$ is magnetically charged, then $x$ and $x \otimes c$ are either both bosonic or both fermionic.
\end{lemma}
\begin{proof}
  As mentioned already (\ref{self-braiding-c}), a superalgebra calculation shows that the Cheshire string itself is fermionic: $\beta_{c,c} \cong e \otimes \id_{c,c}$. The behaviour of the identity component $\Sigma\Omega\cB$ under endofunctors is fully determined by the behaviour of the fermion $e$. In particular, the full braiding map $\beta_{-,x} \circ \beta_{x,-} : \Sigma\Omega\cB \to \Sigma\Omega\cB$ is the identity if $x$ is magnetically neutral, and acts as the 1-form symmetry $(-1)^F$ if $x$ is magnetically charged. It is a fun superalgebra calculation that the automoprhism $(-1)^F$ of $\Cliff(1)$, when converted into a bimodule, is isomorphic to the parity-reversal of the identity bimodule. In other words, the natural automorphism $(-1)^F$ of $\Sigma\SVec$, when evaluated at the Cheshire string $c$, gives $e \otimes \id_c$:
  \begin{equation}\label{braidinglemma1}
    \beta_{c,x} \circ \beta_{x,c} \cong \begin{cases} \id_{x,c}, & \text{if $x$ is magnetically neutral,} \\
    e \otimes \id_{x,c}, & \text{if $x$ is magnetically charged.} \end{cases}
  \end{equation}
  Finally, if $x$ and $y$ are any invertible objects in a braided monoidal $n$-category, then
  \begin{equation}\label{braidinglemma2}
   \beta_{x\otimes y,y \otimes x} \cong \beta_{x,x} \otimes \beta_{y,y} \otimes (\beta_{x,y} \circ \beta_{y,x})
  \end{equation}
  up to identity maps. (The isomorphisms in (\ref{braidinglemma1}--\ref{braidinglemma2}) are not quite canonical in general, but the statement only requires knowing $\beta$ up to isomorphism.)
\end{proof}

But in fact a much stronger statement holds:
\begin{proposition}\label{prop.magbos}
  All magnetically charged strings in $\cB$ are bosonic.
\end{proposition}
\begin{proof}
  Let $x$ be a magnetic string. It selects a homomorphism $\bZ \to \{\text{strings}\}$. Consider the full braided sub-2-groupoid $\cG$ of $\cB^\times$ tensor-generated by the strings $\one,c,x$ and of course the 1-morphisms $1,e$ and 2-morphisms $\bC^\times$. This groupoid lives over the cyclic subgroup of $\{\text{strings}\}$ generated by $x$; pulling it back along the homomorphism $\bZ \to \{\text{strings}\}$ gives a braided 2-groupoid of shape
  $$ \cG = (\bC^\times[2].\{1,e\}[1].\{\one,c\}[0]).\bZ^f[0],$$
  or in other words a space
  $$ \rB^2 \cG = (\bC^\times[4].\{1,e\}[3].\{\one,c\}[2]).\bZ^f[2].$$
  The supscript $f$ reminds that the base $\bZ$ acts by nontrivial 1-form symmetries of the fibre $(\Sigma\SVec)^\times = \bC^\times[2].\{1,e\}[1].\{\one,c\}[0]$.
  
  As in Remark~\ref{remark.fer.SS}, extensions $\cG$ of this shape can be classified by computing a twisted supercohomology group $\SH^5(\bZ^f[2])$, which can be computed by a spectral sequence. 
  $$ \H^i(\bZ[2]; \SH^j(\pt)) \Rightarrow \SH^{i+j}(\bZ^f[2]).$$
  Note that the Eilenberg--Mac Lane space $\bZ[2] = K(\bZ,2)$ is nothing but $\bC\bP^\infty$, which has a single cell in each even degree. Thus
  the $E_2$ page is very easy to write down. With the nonzero $\d_2$ differentials drawn in, it reads:
  $$
  \begin{tikzpicture}[anchor=base]
\path
(0,0) node {$\bC^\times$} ++(.75,0) node {$0$} ++(.75,0) node (b2) {$\bC^\times$} ++(.75,0) node {$0$} ++(.75,0) node {$\bC^\times$} ++(.75,0) node {$0$} ++(.75,0) node (e2) {$\bC^\times$} ++(.75,0) node {$\cdots$}
(0,.5) node (b1) {$\bZ_2$} ++(.75,0) node {$0$} ++(.75,0) node (a2) {$\bZ_2$} ++(.75,0) node {$0$} ++(.75,0) node (e1) {$\bZ_2$} ++(.75,0) node (c2) {$0$} ++(.75,0) node (d2) {$\cdots$}
(0,1) node (a1) {$\bZ_2$} ++(.75,0) node {$0$} ++(.75,0) node {$\bZ_2$} ++(.75,0) node (c1) {$0$} ++(.75,0) node (d1) {$\bZ_2$} ++(.75,0) node {$\cdots$} 
;
\draw[->] (-.75,-.125) -- ++(6,0);
\draw[->] (-.4,-.5) -- ++(0,2);
\path  (0,-.5) node {$0$} ++(.75,0) node {$1$} ++(.75,0) node {$2$} ++(.75,0) node {$3$} ++(.75,0) node {$4$} ++(.75,0) node {$5$} ++(.75,0) node {$6$} ++(.75,0) node {$i$}
(-.75,0) node {$0$} ++(0,.5) node {$1$} ++(0,.5) node {$2$} ++(0,.5) node {$j$}
;
\draw[thick] 
(a1.mid) -- (a2.mid)
(b1.mid) -- (b2.mid)
(d1.mid) -- (d2.mid)
(e1.mid) -- (e2.mid)
;
\path (b1) ++ (-2,0) node {$E_2^{i,j} =$};
\end{tikzpicture}
  $$
  The self-braiding $\beta_{x,x}$ corresponds to the $\bZ_2$ in bidegree $(4,1)$. The fact that that $\bZ_2$ supports a nonzero differential is equivalent to the statement in the Proposition.
  
  In fact, this calculation shows $\SH^5(\bZ^f[2])$ is trivial, meaning that there is a unique isomorphism class of extensions of shape $(\bC^\times[4].\{1,e\}[3].\{\one,c\}[2]).\bZ^f[2]$.
\end{proof}

\begin{remark}
  It is worth emphasizing that spectral sequence calculations like in the above proof are nothing but a clean way of working with braiding data, etc., in higher monoidal groupoids.
\end{remark}

\begin{corollary}\label{cor.undetection}
  Suppose $x$ is a bosonic magnetically neutral string, and $y$ is a magnetically charged string. Then $x$ and $y$ link trivially: the full braiding $\beta_{y,x} \circ \beta_{x,y}$ is isomorphic to $\id_{x\otimes y}$.
\end{corollary}
\begin{proof}
  Otherwise, equation~(\ref{braidinglemma2}) would say that one of the magnetically-charged strings $y$ and $x\otimes y$ 
  is fermionic and the other is bosonic. But this contradicts Proposition~\ref{prop.magbos}.
\end{proof}

\begin{proof}[Proof of the \hyperlink{mainstatement}{Main Theorem}]
  Suppose $\cB$ is a 3+1D topological order with $\Omega\cB \cong \SVec$. Then $\cB$ contains a magnetic string $m$, as otherwise the electron $e$ would be undetectable. If $\pi_0 \cB$ is not just $\bZ_2$, then $\cB$ contains a nontrivial magnetically neutral bosonic string $x$. Because it is magnetically neutral, it is not detected by linking with the fermion $e$ (or the Cheshire string $c$). By Corollary~\ref{cor.undetection}, it is not detected by linking with any magnetically charged string. Furthermore, if it were detected by linking with a magnetically neutral string $y$, then it would be detected by linking with the magnetically charged string $y \otimes m$, again contradicting Corollary~\ref{cor.undetection}. So $x$ is undetectable, violating the remote detectability axiom of topological orders: $x$ would lift to a nontrivial element in the M\"uger centre $\cZ_{(2)}(\cB)$.
  
  Thus $\pi_0 \cB = \{\one,m\} = \bZ_2^F$. The calculations from \S\ref{subsec.fer} then imply that $\cB \cong \cS$ or $\cT$.
\end{proof}

\subsection{Second proof of the Main Theorem: categorified Galois descent} \label{subsec.proof2}

For my second proof of the \hyperlink{mainstatement}{Main Theorem}, I will rely on the analysis of $\Aut(\cS)$ from \S\ref{subsec.autos}, but I will not otherwise require the calculations from \S\ref{subsec.fer}.

The strategy is the following. Given a 3+1D topological order $\cB$, consider the super topological order $\s\cB := \cB \otimes \Sigma\SVec$ given by base-changing along the ``field extension'' $\Vec \to \SVec$; I will call such base-change the \define{superification} of $\cB$. By definition, a 3+1D \define{super topological order} is a nondegenerate braided fusion $2$-\define{supercategory}, meaning an $2$-category enriched over $\SVec$; ``nondegenerate'' means that its enriched M\"uger centre is  $\Sigma\SVec$ (as a $2$-supercategory).

The Tannakian formalism of super fibre functors from \cite{MR1944506} owes its strength to the fact that $\Vec \to \SVec$ is a \define{categorified Galois extension} \cite{MR3623677}. By this I mean in particular that $\Vec$ arises as the fixed points of the 1-form Galois group $\bZ_2^F[1] = \Aut(\SVec) =: \mathrm{Gal}(\SVec/\Vec)$ acting on $\SVec$. Furthermore, \define{Galois descent} says that the algebra of bosonic higher categories is ``the same'' as the algebra of $\bZ_2^F[1]$-equivariant higher supercategories.

For example, any bosonic topological order $\cB$ can be recovered as the fixed points of the $\bZ_2^F[1]$-action on $\s\cB$ induced from the base change $\s\cB = \cB \otimes \Sigma\SVec$. Thus, to classify bosonic topological orders with some property, it suffices to classify supertopological orders with that property, together with their possible Galois actions.
This is a useful and highly generalizable strategy, and is the reason that I give this second proof of the \hyperlink{mainstatement}{Main Theorem}.

Now suppose that $\cB$ has a unique nontrivial particle, and that it is an emergent fermion.
Since I want to work over the ``base field'' $\SVec$, I will not use the same name for $\Omega\cB$. Rather, I will call it $\Omega\cB \cong \Rep(\bZ_2,\epsilon)$, where $\epsilon$ is the nontrivial element in $\bZ_2$, and the notation follows the Tannakian formalism of \cite{MR1944506}: given a finite group $G$ and a central order-2 element $\epsilon \in Z(G)$, $\Rep(G,\epsilon)$ denotes the (bosonic!)\ symmetric fusion 1-category of super representations of $G$ on which $\epsilon$ acts by $(-1)^F$. 
Then $\Omega \s\cB \cong \SVec \otimes \Rep(\bZ_2,\epsilon)$, and it is not hard to show in general that $\s\Rep(G,\epsilon) := \SVec \otimes \Rep(G,\epsilon) \cong \s\Rep(G)$ is independent of the choice $\epsilon$. Indeed, the choice $\epsilon \in Z(G) = \pi_1 \Aut_{\SVec}(\s\Rep(G))$ is nothing but a Galois descent datum to recover $\Rep(G,\epsilon)$ from its superification $\s\Rep(G)$.

The classification argument from \cite{PhysRevX.8.021074} reviewed in Remark~\ref{LKWargument} applies just as well in the super case as in the bosonic case:
\begin{theorem}[{\cite{PhysRevX.8.021074}, \cite[Corollary V.4]{me-TopologicalOrders}}]\label{thm.superLKW}
  Suppose $\s\cB$ is a 3+1D super topological order with $\Omega \s\cB \cong \s\Rep(G)$. Then $\s\cB \cong \cZ_{(1)}(\SVec^\omega[G])$ for some supercohomology class $\omega \in \SH^4(G[1])$. (Here $\cZ_{(1)}$ denotes the super Drinfel'd centre.) \qed
\end{theorem}

But
$ \SH^4(\bZ_2[1]) = 0.$ This can be computed for example by an Atiyah--Hirzebruch spectral sequence $\H^i(\bZ_2[1]; \SH^j(\pt)) \Rightarrow \SH^{i+j}(\bZ_2[1])$, whose $\d_2$ differentials are $\Sq^2 : \SH^2(\pt) = \bZ_2 \to \SH^1(\pt) = \bZ_2$ and $(-1)^{\Sq^2} : \SH^1(\pt) = \bZ_2 \to \SH^0(\pt) = \bC^\times$. The $E_2$ page with $\d_2$ differentials drawn in reads:
$$
  \begin{tikzpicture}[anchor=base]
\path
(0,0) node {$\bC^\times$} ++(.75,0) node {$\bZ_2$} ++(.75,0) node {$0$} ++(.75,0) node {$\bZ_2$} ++(.75,0) node {$0$} ++(.75,0) node (c2) {$\bZ_2$} ++(.75,0) node  {$\dots$}
(0,.5) node (b0) {$\bZ_2$} ++(.75,0) node {$\bZ_2$} ++(.75,0) node {$\bZ_2$} ++(.75,0) node (c1) {$\bZ_2$} ++(.75,0) node (a2) {$\bZ_2$} ++(.75,0) node (b2) {$\dots$}
(0,1) node {$\bZ_2$} ++(.75,0) node {$\bZ_2$} ++(.75,0) node (a1) {$\bZ_2$} ++(.75,0) node (b1) {$\bZ_2$} ++(.75,0) node  {$\dots$}
;
\draw[->] (-.75,-.125) -- ++(5.25,0);
\draw[->] (-.4,-.5) -- ++(0,2);
\path  (0,-.5) node {$0$} ++(.75,0) node {$1$} ++(.75,0) node {$2$} ++(.75,0) node {$3$} ++(.75,0) node {$4$} ++(.75,0) node {$5$} ++(.75,0)  node {$i$}
(-.75,0) node {$0$} ++(0,.5) node {$1$} ++(0,.5) node {$2$} ++(0,.5) node {$j$}
;
\draw[thick] 
(a1.mid) -- (a2.mid)
(b1.mid) -- (b2.mid)
(c1.mid) -- (c2.mid)
;
\path (b0) ++ (-2,0) node {$E_2^{i,j} =$};
\end{tikzpicture}
$$

Thus  $\s\cB$ is independent of $\cB$, provided only that $\Omega\cB \cong \Rep(\bZ_2)$ or $\Rep(\bZ_2,\epsilon)$. In particular:
$$ \s\cB \cong \s\cR \cong \s\cS \cong \dots.$$
In particular, to classify the choices for $\cB$, it suffices to pick one of these ``understood'' topological orders, for instance $\cS$, and work out the automorphism group and possible Galois descent data on~$\s\cS$.

\begin{remark}
  Superification preserves the set of components. Thus $\pi_0 \cB \cong \pi_0\s\cB \cong \pi_0 \s\cS \cong \pi_0 \cS \cong \bZ_2$, and so one could finish the proof by citing the calculations of \S\ref{subsec.fer}. I will instead analyze the Galois descent data directly.
\end{remark}

Since $\s\cS$ is a super topological order, its $\SVec$-linear automorphism group $\Aut_{\SVec}(\s\cS)$ can be analyzed as in \S\ref{subsec.autos}. In particular, there is a super analogue of the ``topological Noether's theorem'' Proposition~\ref{prop.topologicalNoether}: there is a fibre sequence
\begin{equation}\label{eqn.sNoether}
 \rB\s\cS^\times \to \rB \Aut_{\SVec}(\s\cS) \to \rB(\Sigma^3\SVec)^\times.\end{equation}
As in \S\ref{subsec.proof3}, set $\SW^n := (\Sigma^{n-1}\SVec)^\times$.  Then (\ref{eqn.sNoether}) unfolds to a long exact sequence:
\begin{multline}\label{nother-LES-super}
  1 \to \bC^\times \to \pi_2 (\s\cS^\times) \to \pi_3 \Aut_{\SVec}(\s\cS) \to \pi_0 \SW^1 \to \pi_1 (\s\cS^\times) \to \pi_2 \Aut_{\SVec}(\s\cS) \\
   \to \pi_0\SW^2
  \to \pi_0 (\s\cS^\times) \to \pi_1 \Aut_\SVec(\s\cS) \to \pi_0\SW^3 \to \dots
\end{multline}
The first arrow $\bC^\times \to \pi_2 \s\cB^\times$ is an isomorphism. To describe the other maps, I will need some notation. As above, write $1,e \in \Omega\cS = \Rep(G,\epsilon)$ for the simple objects before superification. These are different from the simple objects $\pi_0 \SW^1$ in the ``base field'' $\SVec$, which I will call $1 := \bC^{1|0}$ and $f := \bC^{0|1}$. With this notation in hand, the arrow $\pi_0(\SVec^\times) = \pi_0\SW^1 \to \pi_1 (\s\cS^\times) = \pi_0 (\s\Rep(\bZ_2,\epsilon)^\times)$ is the inclusion
$$ \{1,f\} \to \{1,f,e,fe\} \cong \bZ_2^2.$$
Similarly, it will be helpful to distinguish the identity-component strings $\one,c\in\cS$ from the ``base field'' strings $\one = \mathrm{Cliff}(0)$ and $a := \mathrm{Cliff}(1)$. In other words, $\pi_0\SW^2 = \pi_0(\Sigma\SVec^\times) = \{1,a\} \cong \bZ_2$.

 It is a slightly nontrivial task to work out $\pi_0 (\s\cS^\times)$. The reason is that to compute $\s\cS = \cS \otimes \Sigma\SVec$, one first forms a ``naive tensor product'' whose objects are just ordered pairs, but then one must Karoubi-complete, and this latter step produces new indecomposable objects, some of which might be invertible.
 In the case at hand, $\pi_0 \s\cS = \bZ_2 = \{1,m\}$, and the question is just to understand the identity component $\Sigma\Omega\s\cS = \Sigma(\s\Rep(\bZ_2,\epsilon))$.
\begin{lemma}
  There are six indecomposable objects in the identity component $\Sigma(\s\Rep(\bZ_2,\epsilon))$: $\one,a,c,ac$, and two that are not invertible.
\end{lemma}
\begin{proof}
  After forgetting the braiding and supercategory structure, there is a monoidal equivalence $\s\Rep(\bZ_2,\epsilon) \cong \Vec[\bZ_2^2]$. Thus the simple objects in $\Sigma(\s\Rep(\bZ_2,\epsilon))$ are indexed by simple algebra objects in $\Vec[\bZ_2^2]$. It is a general fact for any finite group $G$ that simple algebra objects in $\Vec[G]$ are indexed by pairs $(K,\alpha)$ where $K \subset G$ is a subgroup and $\alpha \in \H^2_{\mathrm{gp}}(K;\bC^\times)$ (see e.g.\ \cite[Example 7.4.10]{EGNO}).
  Specifically, writing $\bC g$ for the simple object of $\Vec[G]$ corresponding to $g \in G$,
   the algebra corresponding to $(K,\alpha)$ is $\bigoplus_{k \in K} \bC k$ made into a twisted group algebra via (any cocycle representative of) $\alpha$.
   
   The group $G = \bZ_2^2$ has five subgroups, and one of them (namely $G$ itself) has two possible choices for $\alpha$.
  It then remains to compute the fusion rules for these algebras, which depends on the braidings. The identity algebra is $\one \leftrightarrow (K=\{1\}, \alpha = \text{triv})$. Two of the nontrivial elements of $G = \bZ_2^2$, namely $f$ and $e$, are fermions in the sense that $\beta_{f,f} = \beta_{e,e} = -1$, and so the  algebras $a = 1 \oplus f$ and $c = 1 \oplus e$, corresponding to two of the $\bZ_2$ subgroups of $\bZ_2^2$, are invertible. Their tensor product is then also invertible: it is one of the algebra structures on $1 \oplus e \oplus f \oplus fe$. On the other hand, $1 \oplus fe$, with its unique algebra structure, is not invertible, since $fe$ is a boson. Finally, the other algebra structure on $1 \oplus e \oplus f \oplus fe$ is isomorphic to the tensor product of the noninvertible algebra $1\oplus fe$ with either of the invertible algebras $a$ and $c$, and so is not invertible.
\end{proof}

Finally, as I will explain in Lemma~\ref{lemma.SWpt}, $\pi_0 \SW^3$ is trivial. Therefore, the long exact sequence (\ref{nother-LES-super}) implies that $\pi_3 \Aut_{\SVec}(\s\cS)$ vanishes and
\begin{equation}\label{autSS}
 \pi_2 \Aut_{\SVec}(\s\cS) \cong \bZ_2 = \{1,e\}, \quad \pi_1 \Aut_\SVec(\s\cS) \cong \bZ_2^2 = \{\one,c,m,m'\}.
\end{equation}
Of course, $\pi_3$ had to vanish because $\s\cS$, being merely a 2-categorical object, has only 2- and lower-form symmetries.
As in the previous sections, $m' := m \otimes c$. The extension $\pi_2[2].\pi_1[1]$ is the same as for $\Aut(\cS)$: $\beta_{c,c} = e$. In other words:
\begin{lemma}\label{lemma.product}
  The canonical map $\Aut(\cS) \to \Aut_{\SVec}(\s\cS)$ is an isomorphism on 1- and 2-form symmetries. \qed
\end{lemma}

Equation (\ref{autSS}) describes the ``$\SVec$-linear'' 1-form and 2-form symmetries of $\s\cS$. (There are also 0-form symmetries that I did not describe.) But Galois descent data is by definition nonlinear: in addition to $\Aut_{\SVec}(\s\cS)$, we also need to know $\Aut_{\Vec}(\s\cS)$. The trick is that any automorphism restricts to an automorphism of the centre $\Sigma\SVec$, and so there is a fibre sequence
\begin{equation}\label{eqn.galext}
 \rB\Aut_{\SVec}(\s\cS) \to \rB\Aut_{\Vec}(\s\cS) \to \rB\Aut_{\Vec}(\SVec) = \rB\operatorname{Gal}(\SVec/\Vec) = \bZ_2^F[2].
\end{equation}
Galois descent data, of course, consists of a choice of splitting of the extension (\ref{eqn.galext}). But the existence of the bosonic form $\cS$ means that this map splits, and moreover Lemma~\ref{lemma.product} means that it splits as a product (and not just a semidirect product). In other words:
\begin{corollary}
  The natural map $\operatorname{Gal}(\SVec/\Vec) \times \Aut(\cS) \to \Aut_\Vec(\s\cS)$ is an isomorphism on 1- and 2-form symmetries. \qed
\end{corollary}
Since there are no nontrivial $0$-form symmetries of $\SVec$, this implies:
\begin{corollary}
  The possible Galois descent data are precisely the orbits of the $\pi_0\Aut_\Vec(\s\cS)$-action on\\
\mbox{}\hfill  
  $ \pi_0 \operatorname{maps}\bigl( \bZ_2^F[2], \underbrace{\{1,e\}[3].\{\one,c,m,m'\}[2]}_{\text{connective cover of }\rB\Aut(\s\cS)}\bigr).$ \hfill\qed
\end{corollary}

\begin{proof}[Proof of the \hyperlink{mainstatement}{Main Theorem}]
A map $\bZ_2^F[2] \to \{1,e\}[3].\{\one,c,m,m'\}[2]$ consists of: a homomorphism $\bZ_2^F \to \{\one,c,m,m'\}$, together with a trivialization of the corresponding class in $\H^4(\bZ_2^F[2]; \{1,e\}) \cong \bZ_2$. If the nontrivial element of $\bZ_2^F$ maps to $x \in \{\one,c,m,m'\}$, then the class that needs to be trivialized is $\beta_{x,x} \in \{1,e\}$, thought of as an element in $\H^4(\bZ_2^F[2]; \{1,e\}) \cong \{1,e\}$. For $x = c$, there is no choice of trivialization (\ref{self-braiding-c}), but there is for $x \in \{1,m,m'\}$, in which case there are $\H^3(\bZ_2^F[2]; \{1,e\}) \cong \bZ_2$ many choices for the trivialization.

This seems to provide $\{1,m,m'\} \times \H^3(\bZ_2^F[2]; \{1,e\}) = 3\times 2 = 6$ options for the Galois descent data. But note that the data must be considered up to the action of $\pi_0\Aut_\Vec(\s\cS) \supset \pi_0 \Aut(\cS) \cong \bZ_{16}$. As explained in Proposition~\ref{prop.mswitch}, the odd elements in this $\bZ_{16}$ exchange the two magnetic strings $m$ and $m'$. Moreover, the different choices in $\H^3(\bZ_2^F[2]; \{1,e\})$ are exchanged if the map $\bZ_2^F \to \{\one,c,m,m'\}$ selects one of the magnetic strings. (This is essentially the same as Lemma~\ref{lemma.autoboson}.)

At the end of the day, there are three choices for the Galois descent data:
\begin{itemize}
  \item The trivial map $\bZ_2^F \to \{\one,c,m,m'\}$, and the trivial trivialization in $\H^3(\bZ_2^F[2]; \{1,e\})$, recovers the topological order $\cS$.
  \item The two nontrivial maps $\bZ_2^F \to \{\one,c,m,m'\}$, and the two choices of trivialization, are all isomorphic. The corresponding bosonic topological $\cB$ order is the fixed-points of the $\bZ_2^F[1]$-action on $\s\cS$. This action assigns a nontrivial charge to the central fermion $f$, and, since it uses the magnetic string, it also assigns a nontrivial charge to the non-central fermion $e$. The only unscreened particle is the boson $ef$, and so $\Omega\cB \cong \Rep(\bZ_2)$ rather than $\Rep(\bZ_2,\epsilon)$. In other words, $\cB \cong \cR$.
  \item The last option is the trivial map $\bZ_2^F \to \{\one,c,m,m'\}$, but with the nontrivial trivialization in $\H^3(\bZ_2^F[2]; \{1,e\})$. It screens the invisible fermion $f$ and the boson $ef$,  leaving the fermion~$e$ unscreened, and so the fixed topological order $\cB$ has $\Omega\cB \cong \Rep(\bZ_2,\epsilon)$ rather than $\Rep(\bZ_2)$.
  
  This truly is a different topological order $\cT$ than the fixed points $\cS$ for the ``trivial'' Galois descent data. (Of course, $\cT$ must be producible from some Galois descent data, and $\cT \not\cong \cS$ by the discussion after (\ref{S-braiding},\ref{T-braiding}).)
   To see the difference in terms of Galois descent data, note first that the 2-form symmetry (of $\s\cS$) sourced by~$e$ has a nontrivial 't Hooft anomaly $(-1)^{\Sq^2 E} \in \H^5(\bZ_2[3]; \bC^\times)$, exactly because~$e$ is a fermion rather than a boson. (The notation is as in \S\ref{subsec.fer}.) This anomaly restricts along the nontrivial map $\Sq^1 \in \H^3(\bZ_2^F[2]; \bZ_2) = \operatorname{maps}(\bZ_2[2],\bZ_2[3])$ to $(-1)^{\Sq^2 \Sq^1 F} \neq 1 \in \H^5(\bZ_2^F[2];\bC^\times)$. In other words, the Galois action that selects $\cT$ is anomalous.
  
  This nontrivial anomaly of the Galois action manifests directly in $\cT$ as the nontrivial anomaly of the magnetic string $m$ described in Remarks~\ref{inthewild} and~\ref{remark.anomaly2}.
  \qedhere
\end{itemize}
\end{proof}

\begin{remark}
  Anomalies of Galois actions are familiar to physicists. Indeed, time-reversing symmetries are $\bC$-antilinear: $\bZ_2^T = \mathrm{Gal}(\bC/\bR)$, and anomalies of $\bZ_2^T$-symmetries are one of the basic ingredients in topological insulators (see e.g.~\cite{1302.6234}).
  
   Mathematicians are also familiar with anomalies of Galois actions. For example, for any finite-dimensional central simple algebra $A$ over a field $\bK$, there is a finite-degree field extension $\bK \to \bL$ and an isomorphism $A\otimes_\bK \bL \cong \mathrm{Mat}(n,\bL)$, the algebra of $n\times n$ matrices over $\bL$. Since $\Aut_\bL(\mathrm{Mat}(n,\bL)) = \mathrm{PGL}(n,\bL)$, $A$ is determined by a Galois descent datum in the twisted cohomology group $\H^1(\rB \mathrm{Gal}(\bL/\bK); \mathrm{PGL}(n,\bL))$. The \define{anomaly} of this Galois symmetry is the class of $A$ in the Brauer group $\mathrm{Br}(\bK)$. Explicitly, there is a ``universal anomaly'' in $\H^2(\rB\mathrm{PGL}(n,\bL);\bL^\times)$ classifying the extension $\mathrm{GL}(n,\bL) = \bL^\times.\mathrm{PGL}(n,\bL)$. Restricting this universal anomaly along the Galois descent datum produces a class in $\H^2(\rB \mathrm{Gal}(\bL/\bK); \bL^\times)$, and $\mathrm{Br}(\bK)$ is the limit of these groups as $\bL$ approaches the separable closure of $\bK$.
\end{remark}

\subsection{Third proof of the Main Theorem: a long exact sequence} \label{subsec.proof3}

The last proof that I will give of the \hyperlink{mainstatement}{Main Theorem} relies on the  classification result \cite[Corollary V.5]{me-TopologicalOrders}. This result extends (and, as mentioned in Remark~\ref{remark.LKW}, corrects) the classifications announced in \cite{PhysRevX.8.021074,PhysRevX.9.021005}, the first of which I reviewed in Remark~\ref{LKWargument}.

I will state this classification as Theorem~\ref{thm.classification}. The statement will require a bit of notation. Much of what I will say below reproduces material from \cite{me-TopologicalOrders}, and in particular Remark~V.2 therein.

As mentioned at the end of \S\ref{subsec.f2c}, the symmetric fusion category $\Vec$ has an infinite sequence of suspensions $\Sigma^\bullet\Vec$ constructed in \cite{GJFcond}. Another name for the fusion $n$-category $\Sigma^{n-1}\Vec$ is ``$n\Vec$'' \cite{2011.02859}; by convention, $0\Vec = \Sigma^{-1}\Vec = \Omega\Vec = \bC$. These categories fit together as a categorical version of a loop spectrum: $\Omega\Sigma^\bullet\Vec = \Sigma^{\bullet-1}\Vec$. In particular, define $\W^n := (n\Vec)^\times = (\Sigma^{n-1}\Vec)^\times$. Then the spaces $\W^\bullet$ fit together into a loop spectrum in the usual sense. I will use the same name for the corresponding generalized cohomology theory: for any space $X$, I will write
$$ \W^\bullet(X) := \pi_0 \operatorname{maps}(X, \W^\bullet).$$
Here and throughout, I will use cohomological degree conventions. The \define{homotopy groups} of a spectrum use the opposite conventions:
$$ \pi_n \W := \W^{-n}(\pt) = \pi_0(\W^{-n}).$$

\begin{remark}
  Note that $\W^0 = \bC^\times$ (with the discrete topology), and so $\W^\bullet$ is coconnective: its positive-homotopical-degree homotopy groups (the negative cohomology of a point) vanish. This allows for many simplifications. For example, coconnectivity implies that all Atiyah--Hirzebruch spectral sequences for $\W$-cohomology  converge strongly \cite{MR1718076}. It also implies that Borel-equivariant cohomology, i.e.\ the cohomology of the classifying spaces, is genuinely equivariant. Finally, coconnectivity allows all the subtleties of Brown's representability theorem to be ignored: when working with non-coconnective spectra, there can be nontrivial maps which induce the zero map on cohomology theories, but these do not occur for coconnective spectra. As such, I will freely use the terms ``generalized cohomology theory'' and ``spectrum'' interchangeably.
\end{remark}

Every object of $\Sigma^n\Vec$ is $\Sigma\cA$ for some fusion $(n{-}1)$-category $\cA$, defined up to Morita equivalence of fusion $(n{-}1)$-categories. (With ``fusion'' replaced by ``multifusion,'' this is essentially true by construction; see e.g.~\cite[Theorem 3.21]{2011.02859}. That one can always select $\cA$ to be fusion requires that $\bC$ is algebraically closed: one needs the fact that the commutative algebra $\Omega^{n-1}\cA$ admits a map to $\bC$.) An object in $\Sigma^n\Vec$ is invertible if and only if the corresponding fusion $(n{-}1)$-category $\cA$ has trivial Drinfel'd centre \cite[Theorem~2]{me-TopologicalOrders}, in which case it is $\cA = \Sigma\cB$ for a nondegenerate braided fusion $n$-category $\cB = \Omega\cA$ \cite[Theorem~4]{me-TopologicalOrders}. In other words:
$$ \pi_0 \W^n = \frac{\{\text{invertible fusion $(n{-}2)$-categories}\}}{\text{Morita equivalence}} = \frac{\{\text{nondeg braided fusion $(n{-}3)$-categories}\}}{\text{Morita equivalence}}.$$
In particular, $\pi_0 \W^4 = \cW$ is the ``bosonic Witt group'' of nondegenerate braided fusion $1$-categories studied in \cite{MR3039775}. This is why I chose the letter ``$\cW$'' for the spectrum. 

\begin{lemma}\label{lemma.Wpt}
The low-degree homotopy groups of $\W$ are:
\begin{enumerate} \setcounter{enumi}{-1}
  \item $\W^0(\pt) = \bC^\times$.
  \item $\W^1(\pt) = 0$, because every $\otimes$-invertible vector space is one-dimensional. 
  \item $\W^2(\pt) = 0$, because every central simple algebra over $\bC$ is Morita-trivial.
  \item $\W^3(\pt) = 0$, because a braided fusion $0$-category is simply a commutative algebra over $\bC$, and the only one with trivial centre is $\bC$ itself.
  \item $\W^4(\pt) = \cW$ is the bosonic Witt group of \cite{MR3039775}.\qed
\end{enumerate}
\end{lemma}

There is also a ``super'' version $\SW$ of the spectrum $\W$:
$$ \SW^\bullet := (\Sigma^{n-1}\SVec)^\times.$$
Its homotopy groups measure the super topological orders up to Morita equivalence. In particular:
\begin{lemma}\label{lemma.SWpt}
\begin{enumerate} \setcounter{enumi}{-1}
  \item $\SW^0(\pt) = \bC^\times$.
  \item $\SW^1(\pt) = \pi_0 \SVec_\bC^\times = \{\bC^{1|0},\bC^{0|1}\} \cong \bZ_2$. Every $\otimes$-invertible super vector space is one-dimensional, but in the super world, that one dimension can be either bosonic or fermionic.
  \item $\SW^2(\pt) = \pi_0 \cat{SAlg}_\bC^\times = \{\bC, \Cliff(1)\} \cong \bZ_2$. Every central simple superalgebra over $\bC$ is isomorphic either to a super matrix algebra (Morita trivial)  or  to an odd Clifford algebra (Morita-equivalent to $\Cliff(1)$).
  \item $\SW^3(\pt) = 0$. The bosonic proof applies without change.
  \item $\SW^4(\pt) = \rS\cW$ is the ``fermionic Witt group'' of \cite{MR3022755}.\qed
\end{enumerate}
\end{lemma}

Furthermore, Theorem~\ref{thm.superLKW} implies:
\begin{lemma}\label{lemma.SW5}
  Every 3+1D super topological order is Morita trivial, i.e.\ $\SW^5(\pt) = 0$. \qed
\end{lemma}

\begin{remark}
  The extended supercohomology $\SH^\bullet$ used in Remark~\ref{remark.fer.SS} is the connective cover of $\SW^\bullet$ in which one only uses the first four homotopy groups. In other words, for any space $X$ there is a map $\SH^\bullet(X) \to \SW^\bullet(X)$ which is an isomorphism in degrees $\bullet < 4$ and an injection in degree $\bullet = 4$.
\end{remark}

Since $\SW$ was built from $\SVec$, it inherits an action by the (purely 1-form) categorified Galois group $\operatorname{Gal}(\SVec/\Vec) = \Aut_\Vec(\SVec) = \bZ_2^F[1]$. Categorified Galois descent then says that $\W^\bullet$ is nothing but the fixed-point spectrum of the action of $\bZ_2^F[1]$ on $\SW^\bullet$. In particular, the $\W^\bullet$-cohomology of a point is precisely the \emph{twisted} $\SW^\bullet$-cohomology of the classifying space $\bZ_2^F[2] = \pt/\bZ_2^F[1]$ of $\operatorname{Gal}(\SVec/\Vec)$. Consistent with Remark~\ref{remark.fer.SS}, I will write  simply $\SW^\bullet(\bZ_2^F[2])$ for this twisted cohomology:
\begin{equation}\label{coh.gal}
 \W^\bullet(\pt) = \SW^\bullet(\bZ_2^F[2]).
\end{equation}
More generally, given a space $X$ with a $\bZ_2^F[1]$-action, I will write $X / \bZ_2^F[1]$ for the corresponding homotopy quotient, and $\SW^\bullet(X / \bZ_2^F[1])$ for its twisted cohomology.

There is an Atiyah--Hirzebruch spectral sequence computing the right-hand side $\SW^\bullet(\bZ_2^F[2])$ of~(\ref{coh.gal}):
\begin{equation}\label{ahss.sw}
 \H^i(\bZ_2^F[2]; \SW^j(\pt)) \Rightarrow \SW^{i+j}(\bZ_2^F[2]).
\end{equation}
Note that $\bZ_2^F[1]$ acts trivially on the 0-form groups $\SH^j(\pt)$.
If the cohomology were untwisted, then the $\d_2$ differentials would be simply the Postnikov k-invariants for $\SW$, which in the bottom degrees are:
$$ (-1)^{\Sq^2} : \SW^1(\pt) \to \SW^0(\pt), \qquad \Sq^2 : \SW^2(\pt) \to \SW^1(\pt).$$
The twisting, however, modifies these. Write $T \in \H^2(\bZ_2^F[2]; \bZ_2)$ for the degree-2 generator; the whole cohomology ring $\H^\bullet(\bZ_2[2]; \bZ_2)$ is freely generated by $T$ under the Steenrod-algebra action. Then the twisted $\d_2$ differentials for the spectral sequence (\ref{ahss.sw}) are:
$$ \d_2 = (-1)^{T + \Sq^2} : \SW^1(\pt) \to \SW^0(\pt), \qquad \d_2 = T + \Sq^2 : \SW^2(\pt) \to \SW^1(\pt).$$
 Thus the $E_2$ page with nonzero $\d_2$ differentials drawn in reads:
\begin{equation*}
\begin{tikzpicture}[anchor=base]
\path
(0,0) node {$\bC^\times$} ++(.75,0) node {$0$} ++(.75,0) node (b2) {$\bZ_2$} ++(.75,0) node {$0$} ++(.75,0) node {$\bZ_4$} ++(.75,0) node {$\bZ_2$}++(.75,0) node (e2) {$\cdots$}
(0,.5) node (b1) {$\bZ_2$} ++(.75,0) node {$0$} ++(.75,0) node (a2) {$\bZ_2$} ++(.75,0) node {$\bZ_2$} ++(.75,0) node (e1) {$\bZ_2$} ++(.75,0) node (c2) {$\bZ_2^2$} ++(.75,0) node (d2) {$\cdots$}
(0,1) node (a1) {$\bZ_2$} ++(.75,0) node {$0$} ++(.75,0) node {$\bZ_2$} ++(.75,0) node (c1) {$\bZ_2$} ++(.75,0) node (d1) {$\bZ_2$} ++(.75,0) node {$\cdots$} 
(0,1.5) node {$0$} ++(.75,0) node {$0$}++(.75,0) node {$0$}++(.75,0) node {$0$}++(.75,0) node {$0$} ++(.75,0) node {$\cdots$} 
(0,2) node {$\rS\cW$}  ++(.75,0) node {$0$}  
(0,2.5) node {$0$}
;
\draw[->] (-.75,-.125) -- ++(5.25,0);
\draw[->] (-.4,-.5) -- ++(0,3.5);
\path  (0,-.5) node {$0$} ++(.75,0) node {$1$} ++(.75,0) node {$2$} ++(.75,0) node {$3$} ++(.75,0) node {$4$} ++(.75,0) node {$5$} ++(.75,0) node {$i$}
(-.75,0) node {$0$} ++(0,.5) node {$1$} ++(0,.5) node {$2$} ++(0,.5) node {$3$} ++(0,.5) node {$4$} ++(0,.5) node {$5$} ++(0,.5) node {$j$}
;
\draw[thick] 
(a1.mid) -- (a2.mid)
(b1.mid) -- (b2.mid)
(c1.mid) -- (c2.mid)
(d1.mid) -- (d2.mid)
(e1.mid) -- (e2.mid)
;
\end{tikzpicture}
\end{equation*}
The differentials supported in bidegrees $(i,j) = (0,1)$ and $(0,2)$ are forced by Lemma~\ref{lemma.Wpt}, which says that in the limit there can be no cohomology of total degree $i+j \in \{1,2,3\}$. Furthermore, in degree $i+j = 4$ the limit of this spectral sequence is the bosonic Witt group $\cW$, and the restriction to bidegree $(i,j) = (0,4)$ is the canonical map  $\cW \to \rS\cW$. This map is known to have kernel of order $16$ \cite[Propositions~5.13 and~5.14]{MR3022755}, and so:
\begin{lemma}\label{lemma.nod3}
All of the $\bZ_2$s and $\bZ_4$ in total degree $4$ survive to the $E_\infty$ page. \qed
\end{lemma}
\begin{remark}\label{remark.ssMME}
  The $\d_4$ differential automatically vanishes.  The only other potential differential in this range of degrees is $\d_5 : \rS\cW = E_5^{0,4} \to \bZ_2 = E_5^{5,0}$. Either it vanishes, in which case $\W^5(\pt) = \bZ_2$ and $\cW \to \rS\cW$ is a surjection, or it does not vanish, in which case $\W^5(\pt) = 0$ and $\operatorname{coker}(\cW \to \rS\cW) \cong \bZ_2$.
  
  I do not know which case occurs, but the good money is on the former. This is because, as pointed out by Dmitri Nikshych (see \cite[Remark~V.2]{me-TopologicalOrders}), the surjectivity of the map $\cW \to \rS\cW$ is equivalent to the conjecture that every slightly-degenerate braided fusion 1-category admits a minimal modular extension, and most experts believe that conjecture.
\end{remark}

I will need one more notion in order to give the complete classification from \cite{me-TopologicalOrders}. For any space $X$, there is a canonical map $\pi : X \to \pt$, and so for any cohomology theory $E$ a restriction map $\pi^* : E^\bullet(\pt) \to E^\bullet(X)$. The \define{reduced}  $E^\bullet$-cohomology of $X$, denoted $\widetilde{E}^\bullet(X)$, is by definition the homotopy cokernel of this map: a cocycle for a class in $\widetilde{E}^n(X)$ is a pair $(\alpha,\beta)$, where $\alpha$ is a cocycle for a class in $\widetilde{E}^{n+1}(\pt)$, and $\beta$ is a degree-$n$ cochain on $X$ solving $\d \beta = \pi^*\alpha$. Fix a point $x \in X$. This choice provides a restriction $E^\bullet(X) \to E^\bullet(\{x\})$, and the composition $E^\bullet(\pt) \to E^\bullet(X) \to E^\bullet(\{x\})$ is the identity. As a result, $E^\bullet(\pt)$ is a direct summand of $E^\bullet(X)$, and is simply the other direct summand:
\begin{equation}\label{eqn.reducessum}
 \widetilde{E}^\bullet(X) = E^\bullet(X) \ominus E^\bullet(\pt).\end{equation}

Now suppose that $E^\bullet$ is a twisted $\cG$-equivariant cohomology theory, for instance $E^\bullet = \SW^\bullet$ and $\cG = \bZ_2^F[1]$. Then there is still a \define{reduced} twisted-equivariant $E^\bullet$-cohomology of any $\cG$-space $X$. As above, let $E^\bullet(X/\cG)$ denote  the unreduced twisted $\cG$-equivariant cohomology of $X$. The reduced cohomology $\widetilde{E}^\bullet(X/\cG)$ is defined to be the homotopy cokernel of the map $\pi^* : E^\bullet(\pt/\cG) \to E^\bullet(X/\cG)$ that restricts along the canonical map $X/\cG \to \pt/\cG$ (i.e.\ the canonical $G$-equivariant map $X \to \pt$). 

If $X$ contains a $\cG$-fixed point $x \in X$ (equivalently, if the map $X/\cG \to \pt/\cG$ splits), then the discussion leading to equation~(\ref{eqn.reducessum}) describes $\widetilde{E}^\bullet(X/\cG)$ as a direct summand of $E^\bullet(X/\cG)$. But in general, there are no such fixed points. Indeed, $\pi^* : E^\bullet(\pt/\cG) \to E^\bullet(X/\cG)$ can fail to be an injection. Rather, in place of equation (\ref{eqn.reducessum}), there is a long exact sequence:
\begin{equation}\label{eqn.LES}
 \dots \to E^\bullet(\pt/\cG) \to E^\bullet(X/\cG) \to \widetilde{E}^\bullet(X/\cG) \to E^{\bullet+1}(\pt/\cG) \to E^{\bullet+1}(X/\cG) \to \dots 
\end{equation}

With these concepts in place, I can now state the promised generalization \cite[Corollary V.5]{me-TopologicalOrders} of the classification described in Remark~\ref{LKWargument}:
\begin{theorem} \label{thm.classification}
  Let $G$ be a finite group. An action of $\bZ_2^F[1]$ on the classifying space $\rB G$ is the same data as a central element $\epsilon \in Z(G)$ such that $\epsilon^2 = 1$; the action is nontrivial if $\epsilon \neq 1 \in G$, in which case $BG$ contains no $\bZ_2^F[1]$-fixed points. Let $\Rep(G,\epsilon)$ denote the corresponding (bosonic) symmetric fusion category: it consists of superrepresentations of $G$ on which $\epsilon$ acts by $(-1)^f$.
  Then the 3+1D topological orders $\cB$ with line operators $\Omega\cB \cong \Rep(G,\epsilon)$ are classified by 
  $\widetilde\SW^4(\rB G/\bZ_2^F[1])$, 
   the {reduced $\bZ_2^F[1]$-twisted-equivariant} degree-$4$ $\SW$-cohomology of $\rB G$. \qed
\end{theorem}

Note that, by \cite{MR1944506}, every symmetric fusion category is equivalent to some $\Rep(G,\epsilon)$. The proof of Theorem~\ref{thm.classification} essentially amounts to a careful study of how categorified Galois descent interacts with Theorem~\ref{thm.superLKW}, which described the classification of super topological orders. When $\epsilon = 1$, one can choose a $\bZ_2^F[1]$-fixed point in $\rB G$, and $\SW^\bullet(\rB G / X) = \W^\bullet(\rB G)$, and so Theorem~\ref{thm.classification} reduces to the bosonic classification described in Remark~\ref{LKWargument}.

\begin{proof}[Proof of the \hyperlink{mainstatement}{Main Theorem}]
  I will apply Theorem~\ref{thm.classification} in the case when $\Rep(G,\epsilon) = \Rep(\bZ_2,\epsilon) = \SVec$; I will write simply ``\text{want}'' for the desired group $\widetilde\SW^4(\rB \bZ_2/\bZ_2^F[1])$. 
  
  The action of $\bZ_2^F[1]$ on $\rB G = \rB \bZ_2 = \bZ_2[1]$ corresponding to the nontrivial $\epsilon$ is the canonical one ``by multiplication,'' and so the quotient is $\rB G/\bZ_2^F[1] = \pt$. Set $E^\bullet = \SW^\bullet$ and $X = \rB G$ and $\cG = \bZ_2^F[1]$. Then the entries $E^\bullet(\pt/\cG)$ and $E^\bullet(X/\cG)$ in the long exact sequence (\ref{eqn.LES}) are
  $$ E^\bullet(\pt/\cG) = \SW^\bullet(\bZ_2^F[2]) = \W^\bullet(\pt), \qquad E^\bullet(X/\cG) = \SW^\bullet(\pt).$$  
  Thus the long exact sequence reads:
  $$ \dots \to \W^4(\pt) \to \SW^4(\pt) \to \text{want} \to \W^5(\pt) \to \SW^5(\pt) \to \dots$$
  Filling in the known values from Lemmas~\ref{lemma.Wpt}, \ref{lemma.SWpt}, and~\ref{lemma.SW5} gives:
  $$ \dots \to \cW \to \rS\cW \to \text{want} \to \W^5(\pt) \to 0 \to \dots$$
  Finally, Remark~\ref{remark.ssMME} says that either $\operatorname{coker}(\cW \to \rS\cW) = \bZ_2$ and $\W^5 = 0$, or $\operatorname{coker}(\cW \to \rS\cW) = 0$ and $\W^5 = \bZ_2$. In either case, the desired group ``\text{want}'' is isomorphic to $\bZ_2$.
\end{proof}

\begin{remark}\label{remark.MMEredux}
  This proof makes clear that if $\pi_0\W^5 = \pi_0(5\Vec)^\times$ is nontrivial, then $\cT$ represents the nontrivial class; contrapositively, if $\cT$ is a Drinfel'd centre, then so is every bosonic 3+1D topological order. Translating through Remarks~\ref{remark.MME} and~\ref{remark.ssMME}, one finds that the minimal modularity conjecture for slightly-degenerate fusion 1-categories is equivalent to the statement that there exists a Morita-nontrivial 3+1D topological order.
  
  Actually, that fact that there is at most one nontrivial class in $\pi_0\W^5$, and that if it exists it is represented by $\cT$, follows directly from the construction in \cite{PhysRevX.9.021005} together with the \hyperlink{mainstatement}{Main Theorem} from this paper. Indeed,  \cite{PhysRevX.9.021005} gives an algorithm to condense any bosonic 3+1D topological order $\cB$ either directly to the vacuum (in which case $\cB$ was Morita trivial) or to a topological order with a single emergent fermion and nothing else. Since, by the \hyperlink{mainstatement}{Main Theorem}, there are only two topological orders with this particle content, one of which ($\cS \cong \cZ_{(1)}(\Sigma\SVec)$) is Morita-trivial, one finds that if a topological order $\cB$ is not Morita trivial, then it must be Morita equivalent to $\cT$.
\end{remark}


\begin{thebibliography}{KLWZZ20b}

\bibitem[AdR20]{MR4073072}
M.~F. Araujo~de Resende.
\newblock A pedagogical overview on 2{D} and 3{D} toric codes and the origin of
  their topological orders.
\newblock {\em Rev. Math. Phys.}, 32(2):2030002, 32, 2020.
\newblock \DOI{10.1142/S0129055X20300022}. \MRnumber{4073072}.
  \arXiv{1712.01258}.

\bibitem[BJSS20]{BJSS2020}
Adrien Brochier, David Jordan, Pavel Safronov, and Noah Snyder.
\newblock Invertible braided tensor categories.
\newblock 2020.
\newblock \arXiv{2003.13812}.

\bibitem[Boa99]{MR1718076}
J.~Michael Boardman.
\newblock Conditionally convergent spectral sequences.
\newblock In {\em Homotopy invariant algebraic structures ({B}altimore, {MD},
  1998)}, volume 239 of {\em Contemp. Math.}, pages 49--84. Amer. Math. Soc.,
  Providence, RI, 1999.
\newblock \DOI{10.1090/conm/239/03597}. \MRnumber{1718076}.

\bibitem[Del02]{MR1944506}
P.~Deligne.
\newblock Cat{\'e}gories tensorielles.
\newblock {\em Mosc. Math. J.}, 2(2):227--248, 2002.
\newblock Dedicated to Yuri I. Manin on the occasion of his 65th birthday.
  \MRnumber{1944506}.

\bibitem[DG20]{2004.11395}
Diego Delmastro and Jaume Gomis.
\newblock Domain walls in 4d {N=1} supersymmetric {Yang-Mills}.
\newblock 2020.
\newblock \arXiv{2004.11395}.

\bibitem[DMNO13]{MR3039775}
Alexei Davydov, Michael M{\"u}ger, Dmitri Nikshych, and Victor Ostrik.
\newblock The {W}itt group of non-degenerate braided fusion categories.
\newblock {\em J. Reine Angew. Math.}, 677:135--177, 2013.
\newblock \MRnumber{3039775}. \arXiv{1009.2117}.

\bibitem[DN20]{2006.08022}
Alexei Davydov and Dmitri Nikshych.
\newblock Braided {Picard} groups and graded extensions of braided tensor
  categories.
\newblock 2020.
\newblock \arXiv{2006.08022}.

\bibitem[DNO13]{MR3022755}
Alexei Davydov, Dmitri Nikshych, and Victor Ostrik.
\newblock On the structure of the {W}itt group of braided fusion categories.
\newblock {\em Selecta Math. (N.S.)}, 19(1):237--269, 2013.
\newblock \DOI{10.1007/s00029-012-0093-3}. \MRnumber{3022755}.
  \arXiv{1109.5558}.

\bibitem[DR18]{Reutter2018}
Christopher~L. Douglas and David~J. Reutter.
\newblock Fusion 2-categories and a state-sum invariant for 4-manifolds.
\newblock 2018.
\newblock \arXiv{1812.11933}.

\bibitem[DSPS20]{DSPS}
Christopher~L. Douglas, Christopher Schommer-Pries, and Noah Snyder.
\newblock Dualizable tensor categories.
\newblock {\em Memoirs of the AMS}, 2020.
\newblock \arXiv{1312.7188}.

\bibitem[EGNO15]{EGNO}
Pavel Etingof, Shlomo Gelaki, Dmitri Nikshych, and Victor Ostrik.
\newblock {\em Tensor categories}, volume 205 of {\em Mathematical Surveys and
  Monographs}.
\newblock American Mathematical Society, Providence, RI, 2015.
\newblock \url{http://www-math.mit.edu/~etingof/egnobookfinal.pdf}.
  MRnumber{3242743}. \DOI{10.1090/surv/205}.

\bibitem[EN17]{1702.02148}
Dominic~V. Else and Chetan Nayak.
\newblock Cheshire charge in (3+1)-dimensional topological phases.
\newblock {\em Phys. Rev. B}, 96(4):045136, 2017.
\newblock \arXiv{1702.02148}. \DOI{10.1103/PhysRevB.96.045136}.

\bibitem[GJF19a]{GJFcond}
Davide Gaiotto and Theo Johnson-Freyd.
\newblock Condensations in higher categories.
\newblock 2019.
\newblock \arXiv{1905.09566}.

\bibitem[GJF19b]{MR3978827}
Davide Gaiotto and Theo Johnson-Freyd.
\newblock Symmetry protected topological phases and generalized cohomology.
\newblock {\em J. High Energy Phys.}, (5):007, 34, 2019.
\newblock \DOI{10.1007/JHEP05(2019)007}. \arXiv{1712.07950}.
  \MRnumber{3978827}.

\bibitem[GK16]{1505.05856}
Davide Gaiotto and Anton Kapustin.
\newblock Spin {TQFT}s and fermionic phases of matter.
\newblock {\em Internat. J. Modern Phys. A}, 31(28n29):1645044, 2016.
\newblock \DOI{10.1142/S0217751X16450445}. \arXiv{1505.05856}.

\bibitem[GKSW15]{MR3321281}
Davide Gaiotto, Anton Kapustin, Nathan Seiberg, and Brian Willett.
\newblock Generalized global symmetries.
\newblock {\em J. High Energy Phys.}, (2):172, front matter+61, 2015.
\newblock \DOI{10.1007/JHEP02(2015)172}. \MRnumber{3321281}. \arXiv{1412.5148}.

\bibitem[GS18]{GwSch18}
Owen Gwilliam and Claudia Scheimbauer.
\newblock Duals and adjoints in higher {M}orita categories.
\newblock 2018.
\newblock \arXiv{1804.10924}.

\bibitem[HZW05]{cond-mat/0411752}
Alioscia Hamma, Paolo Zanardi, and Xiao-Gang Wen.
\newblock String and membrane condensation on 3d lattices.
\newblock {\em Phys. Rev. B}, 72(3):035307, 2005.
\newblock \arXiv{cond-mat/0411752}. \DOI{10.1103/PhysRevB.72.035307}.

\bibitem[JF17]{MR3623677}
Theo Johnson-Freyd.
\newblock Spin, statistics, orientations, unitarity.
\newblock {\em Algebr. Geom. Topol.}, 17(2):917--956, 2017.
\newblock \DOI{10.2140/agt.2017.17.917}. \MRnumber{3623677}.
  \arXiv{1507.06297}.

\bibitem[JF20a]{me-Heisenberg}
Theo Johnson-Freyd.
\newblock Heisenberg-picture quantum field theory.
\newblock In {\em Representation Theory, Mathematical Physics and Integrable
  Systems}, Progr. Math. Birkh\"auser Boston, 2020.
\newblock Volume in honor of Kolya Reshetikhin. \arXiv{1508.05908}.

\bibitem[JF20b]{me-TopologicalOrders}
Theo Johnson-Freyd.
\newblock On the classification of topological orders.
\newblock 2020.
\newblock \arXiv{2003.06663}.

\bibitem[JFS17]{JFS}
Theo Johnson-Freyd and Claudia Scheimbauer.
\newblock {(Op)lax} natural transformations, twisted field theories, and ``even
  higher'' {M}orita categories.
\newblock {\em Adv. Math.}, 307:147--223, 2 2017.
\newblock \DOI{10.1016/j.aim.2016.11.014}. \arXiv{1502.06526}.

\bibitem[JFY20]{JFYu}
Theo Johnson-Freyd and Matthew Yu.
\newblock Fusion 2-categories with no line operators are grouplike.
\newblock 2020.

\bibitem[Kit97]{Kitaev-QCM}
Alexei Kitaev.
\newblock Quantum error correction with imperfect gates.
\newblock In Osamu Hirota, A.S. Holevo, and C.M. Caves, editors, {\em Quantum
  Communication, Computing, and Measurement}, pages 181--188. Springer US,
  1997.
\newblock \DOI{10.1007/978-1-4615-5923-8}.

\bibitem[Kit06]{MR2200691}
Alexei Kitaev.
\newblock Anyons in an exactly solved model and beyond.
\newblock {\em Ann. Physics}, 321(1):2--111, 2006.
\newblock \DOI{10.1016/j.aop.2005.10.005}. \MRnumber{2200691}.
  \arXiv{cond-mat/0506438}.

\bibitem[KLWZZ20a]{KLWZZ2}
Liang Kong, Tian Lan, Xiao-Gang Wen, Zhi-Hao Zhang, and Hao Zheng.
\newblock Algebraic higher symmetries and categorical symmetry -- a holographic
  entanglement view of symmetry.
\newblock 2020.
\newblock \arXiv{2005.14178}.

\bibitem[KLWZZ20b]{KLWZZ}
Liang Kong, Tian Lan, Xiao-Gang Wen, Zhi-Hao Zhang, and Hao Zheng.
\newblock Classification of topological phases with finite internal symmetries
  in all dimensions.
\newblock 2020.
\newblock \arXiv{2003.08898}.

\bibitem[KTZ19]{1905.04644}
Liang Kong, Yin Tian, and Shan Zhou.
\newblock The center of monoidal 2-categories in {3+1D} {D}ijkgraaf-{W}itten
  theory.
\newblock 2019.
\newblock \arXiv{1905.04644}.

\bibitem[KTZ20]{2009.06564}
Liang Kong, Yin Tian, and Zhi-Hao Zhang.
\newblock Defects in the 3-dimensional toric code model form a braided fusion
  2-category.
\newblock 2020.
\newblock \arXiv{2009.06564}.

\bibitem[KW14]{1405.5858}
Liang Kong and Xiao-Gang Wen.
\newblock Braided fusion categories, gravitational anomalies, and the
  mathematical framework for topological orders in any dimensions.
\newblock 2014.
\newblock \arXiv{1405.5858}.

\bibitem[KWZ15]{KWZ1}
Liang Kong, Xiao-Gang Wen, and Hao Zheng.
\newblock Boundary-bulk relation for topological orders as the functor mapping
  higher categories to their centers.
\newblock 2015.
\newblock \arXiv{1502.01690}.

\bibitem[KWZ17]{KWZ2}
Liang Kong, Xiao-Gang Wen, and Hao Zheng.
\newblock Boundary-bulk relation in topological orders.
\newblock {\em Nuclear Phys. B}, 922:62--76, 9 2017.
\newblock \DOI{10.1016/j.nuclphysb.2017.06.023}. \arXiv{1702.00673}.

\bibitem[KZ20]{2011.02859}
Liang Kong and Hao Zheng.
\newblock Categories of topological orders {I}.
\newblock 2020.
\newblock \arXiv{2011.02859}.

\bibitem[Lev13]{PhysRevX.3.021009}
Michael Levin.
\newblock Protected edge modes without symmetry.
\newblock {\em Phys. Rev. X}, 3:021009, May 2013.
\newblock \DOI{10.1103/PhysRevX.3.021009}. \arXiv{1301.7355}.

\bibitem[LKW17]{MR3613518}
Tian Lan, Liang Kong, and Xiao-Gang Wen.
\newblock Modular extensions of unitary braided fusion categories and {$2+1{\rm
  D}$} topological/{SPT} orders with symmetries.
\newblock {\em Comm. Math. Phys.}, 351(2):709--739, 2017.
\newblock \DOI{10.1007/s00220-016-2748-y}. \MRnumber{3613518}.
  \arXiv{1602.05936}.

\bibitem[LKW18]{PhysRevX.8.021074}
Tian Lan, Liang Kong, and Xiao-Gang Wen.
\newblock Classification of $\mathbf{(}3+1\mathbf{)}\mathrm{D}$ bosonic
  topological orders: The case when pointlike excitations are all bosons.
\newblock {\em Phys. Rev. X}, 8:021074, Jun 2018.
\newblock \DOI{10.1103/PhysRevX.8.021074}. \arXiv{1704.04221}.

\bibitem[LW19]{PhysRevX.9.021005}
Tian Lan and Xiao-Gang Wen.
\newblock Classification of $3+1\mathrm{D}$ bosonic topological orders (ii):
  The case when some pointlike excitations are fermions.
\newblock {\em Phys. Rev. X}, 9:021005, Apr 2019.
\newblock \DOI{10.1103/PhysRevX.9.021005}. \arXiv{1801.08530}.

\bibitem[M{\"{u}}g03]{MR1990929}
Michael M{\"{u}}ger.
\newblock On the structure of modular categories.
\newblock {\em Proc. London Math. Soc. (3)}, 87(2):291--308, 2003.
\newblock \DOI{10.1112/S0024611503014187}. \arXiv{math/0201017}.
  \MRnumber{1990929}.

\bibitem[Sch14]{ScheimbauerThesis}
Claudia Scheimbauer.
\newblock {\em Factorization Homology as a Fully Extended Topological Field
  Theory}.
\newblock PhD thesis, ETH Z{\"u}rich, 2014.
\newblock
  \url{https://www.research-collection.ethz.ch/handle/20.500.11850/154981}.

\bibitem[SP09]{Schommer-Pries:thesis}
Christopher~J. Schommer-Pries.
\newblock {\em The Classification of Two-Dimensional Extended Topological Field
  Theories}.
\newblock PhD thesis, University of California, Berkeley, 2009.
\newblock \arXiv{1112.1000}. \MRnumber{MR2534210}.

\bibitem[Tho15]{MR3321301}
Ryan Thorngren.
\newblock Framed {W}ilson operators, fermionic strings, and gravitational
  anomaly in 4d.
\newblock {\em J. High Energy Phys.}, (2):152, front matter+14, 2015.
\newblock \DOI{10.1007/JHEP02(2015)152}. \MRnumber{3321301}. \arXiv{1404.4385}.

\bibitem[Wal64]{MR167498}
C.~T.~C. Wall.
\newblock Graded {B}rauer groups.
\newblock {\em J. Reine Angew. Math.}, 213:187--199, 1963/64.
\newblock \DOI{10.1515/crll.1964.213.187}. \MRnumber{167498}.

\bibitem[Weg71]{MR289087}
Franz~J. Wegner.
\newblock Duality in generalized {I}sing models and phase transitions without
  local order parameters.
\newblock {\em J. Mathematical Phys.}, 12:2259--2272, 1971.
\newblock \DOI{10.1063/1.1665530}.

\bibitem[WG17]{WangGu2017}
Qing-Rui Wang and Zheng-Cheng Gu.
\newblock Towards a complete classification of fermionic symmetry protected
  topological phases in 3d and a general group supercohomology theory.
\newblock 2017.
\newblock \arXiv{1703.10937}.

\bibitem[WS13]{1302.6234}
Chong Wang and T.~Senthil.
\newblock Boson topological insulators: A window into highly entangled quantum
  phases.
\newblock {\em Phys. Rev. B}, 87(23):235122, 2013.
\newblock \arXiv{1302.6234}. \DOI{10.1103/PhysRevB.87.235122}.

\end{thebibliography}

\newcommand{\etalchar}[1]{$^{#1}$}

\end{document}